\documentclass[11pt]{amsart}

\usepackage{amsfonts,amssymb,enumerate,color,bm,hhline,makecell,array}
\usepackage{array}
\usepackage{mathrsfs}
\usepackage{comment}
\usepackage[all,ps,cmtip]{xy} 
\usepackage[normalem]{ulem}
\usepackage[colorlinks=true,citecolor=blue, urlcolor=blue, linkcolor=blue, pagebackref]{hyperref}
\usepackage{tikz-cd}
\usepackage{amscd}
\usepackage[shortlabels]{enumitem}
\usepackage{tensor}
\usepackage{tikz}
\usetikzlibrary{matrix,arrows}
\usepackage{extarrows}
\usepackage{subcaption}
\usepackage{mathtools}
\usepackage{comment}

\def\C{\ensuremath{\mathbb{C}}}

\def\H{\ensuremath{\mathbb{H}}}

\def\P{\ensuremath{\mathbb{P}}}
\def\Q{\ensuremath{\mathbb{Q}}}

\def\Z{\ensuremath{\mathbb{Z}}}

\def\cC{\ensuremath{\mathcal C}}

\def\cE{\ensuremath{\mathcal E}}
\def\cF{\ensuremath{\mathcal F}}
\def\cG{\ensuremath{\mathcal G}}
\def\cH{\ensuremath{\mathcal H}}
\def\cK{\ensuremath{\mathcal K}}
\def\cL{\ensuremath{\mathcal L}}
\def\cM{\ensuremath{\mathcal M}}

\def\cO{\ensuremath{\mathcal O}}
\def\cP{\ensuremath{\mathcal P}}

\def\cX{\ensuremath{\mathcal X}}

\def\cZ{\ensuremath{\mathcal Z}}

\def\vv{\ensuremath{\mathbf v}}

\def\fM{\mathfrak M}

\def\fg{\mathfrak g}

\def\phi{{\varphi}}

\DeclareMathOperator{\ch}{ch}

\DeclareMathOperator{\Db}{D\textsuperscript{\rm b}}

\DeclareMathOperator{\Ext}{Ext}

\DeclareMathOperator{\Hom}{Hom}

\DeclareMathOperator{\Ker}{Ker}

\DeclareMathOperator{\NS}{NS}
\DeclareMathOperator{\Pic}{Pic}

\DeclareMathOperator{\rk}{rk}
\DeclareMathOperator{\Supp}{Supp}

\DeclareMathOperator{\td}{td}
\DeclareMathOperator{\bwed}{\bigwedge\nolimits}

\def\Coh{\mathop{\mathrm{Coh}}\nolimits}

\def\ext{\mathop{\mathrm{ext}}\nolimits}

\def\RHom{\mathop{\mathrm{RHom}}\nolimits}
\def\id{\mathop{\mathrm{id}}\nolimits}

\def\Ker{\mathop{\mathrm{Ker}}\nolimits}
\def\NS{\mathop{\mathrm{NS}}\nolimits}

\def\rk{\mathop{\mathrm{rk}}}

\def\Tr{\mathop{\mathrm{Tr}}\nolimits}

\def\SH{\mathrm{SH}}
\def\Sym{\mathrm{Sym}}
\def\At{\mathrm{At}}

\def\isom{\simeq}

\newcommand{\HK}{hyper-K\"ahler }
\newcommand{\wt}{\widetilde}
\newcommand{\RcHom}{\mathrm{R}\mathcal{H}\!{\it om}}
\def\sfq{\ensuremath{\mathsf q}}

\newcommand{\mor}[1][]{\xrightarrow{#1}}

\newtheorem{Thm}{Theorem}[section]
\newtheorem{Prop}[Thm]{Proposition}

\newtheorem{Lem}[Thm]{Lemma}
\newtheorem{Cor}[Thm]{Corollary}

\newtheorem*{Con*}{Conjecture}
\newtheorem*{Ques*}{Question}
\theoremstyle{definition}
\newtheorem{Defi}[Thm]{Definition}
\newtheorem{Rem}[Thm]{Remark}
\newtheorem{Ex}[Thm]{Example}

\makeatletter
\@namedef{subjclassname@2020}{%
  \textup{2020} Mathematics Subject Classification}
\makeatother

\makeatletter
\def\@tocline#1#2#3#4#5#6#7{\relax
  \ifnum #1>\c@tocdepth % then omit
  \else
    \par \addpenalty\@secpenalty\addvspace{#2}%
    \begingroup \hyphenpenalty\@M
    \@ifempty{#4}{%
      \@tempdima\csname r@tocindent\number#1\endcsname\relax
    }{%
      \@tempdima#4\relax
    }%
    \parindent\z@ \leftskip#3\relax \advance\leftskip\@tempdima\relax
    \rightskip\@pnumwidth plus4em \parfillskip-\@pnumwidth
    #5\leavevmode\hskip-\@tempdima
      \ifcase #1
       \or\or \hskip 1em \or \hskip 2em \else \hskip 3em \fi%
      #6\nobreak\relax
    \dotfill\hbox to\@pnumwidth{\@tocpagenum{#7}}\par
    \nobreak
    \endgroup
  \fi}
\makeatother

\allowdisplaybreaks

\usepackage{enumitem}
\setlist[itemize]{noitemsep,nolistsep}
\setlist[enumerate]{noitemsep,nolistsep}
\usepackage{colortbl}
\usepackage{enumerate}

\begin{document}

\definecolor{xdxdff}{rgb}{0.49019607843137253,0.49019607843137253,1}
\definecolor{uuuuuu}{rgb}{0.26666666666666666,0.26666666666666666,0.26666666666666666}
\definecolor{ffqqqq}{rgb}{1,0,0}

\title{Towards a modular construction of OG10}

\author[Alessio Bottini]{Alessio Bottini}

\address{Dipartimento di Matematica, Universit\`{a} di Roma Tor Vergata, Via della Ricerca Scientifica 1, 00133, Roma, Italia}
\address{Universit\'e Paris-Saclay,
CNRS, Laboratoire de Math\'ematiques d'Orsay,
Rue Michel Magat, B\^at. 307, 91405 Orsay, France}
\curraddr{Mathematisches Institut, Universität Bonn, Endenicher Allee 60, 53115 Bonn, Germany}
\email{bottini@math.uni-bonn.de}

%MathSubjClass2020
\makeatletter
\@namedef{subjclassname@2020}{%
  \textup{2020} Mathematics Subject Classification}
\makeatother

\begin{abstract}
We construct the first example of a stable hyperholomorphic vector bundle of rank five on every hyper-K\"ahler manifold of $\mathrm{K3}^{[2]}$-type whose deformation space is smooth of dimension ten. 
Its moduli space is birational to a hyper-K\"ahler manifold of type OG10.
This provides evidence for the expectation that moduli spaces of sheaves on a hyper-K\"ahler could lead to new examples of hyper-K\"ahler manifolds.
\end{abstract}
%We give a new example of a stable atomic (in particular hyperholomorphic) vector bundle of rank five on every hyper-K\"ahler manifold of $\mathrm{K3}^{[2]}$-type. We show that its deformation space is smooth of dimension ten. We study the moduli space of its semistable deformations and show that its birational to a hyper-K\"ahler manifold of type OG10. There is a natural closed 2-form on the smooth locus; we show that it is preserved by the birational map.

\maketitle
\setcounter{tocdepth}{1}
\tableofcontents

%%%%%%%%%%%%%%%%%%%%%%%%%%%%%%%%%%%%%

\section{Introduction}
\subsection{Background and motivation}
Hyper-K\"ahler manifolds are a well-studied class of K\"ahler manifolds. Interest in them originates Beauville--Bogomolov decomposition Theorem which shows that they are building blocks for K\"ahler manifolds with torsion first Chern class. Despite having a well-developed general theory, culminating in Verbitsky's Global Torelli Theorem \cite{verbitsky13}, few examples are known. Moreover, all known examples arise, up to deformation, from smooth moduli spaces of sheaves on K3 (or Abelian) surfaces \cite{ogrady97, Yoshioka_main}, or desingularizations of singular moduli spaces \cite{ogrady99, ogrady03}. 

In a quest to generalize to higher dimensions the properties of vector bundles on K3 surfaces, which allow for the rich geometry of their moduli spaces, O'Grady \cite{ogrady22} introduced the notion of modular sheaves.
A torsion-free sheaf $F$ on a \HK manifold $X$ is called modular if its discriminant satisfies a certain numerical condition, see Definition \ref{def:modular}. 

This property is satisfied, for example, if $\Delta(F)$ remains of type $(2,2)$ along all deformations of $X$. 
In this case, a celebrated result due to Verbitsky \cite[Theorem $3.19$]{verbitsky99} for locally free sheaves, and later generalized by Markman \cite[Corollary $6.12$]{markman20} to reflexive sheaves, says that if $F$ is also slope-stable, then it deforms along any K\"ahler deformation of $X$. 

As described in \cite{ogrady21} and \cite{ogrady22}, the underlying motivation for the notion of modularity is to try to extend to higher dimensions the proof that moduli spaces of sheaves on K3 surfaces are hyper-K\"ahler manifolds. 
Specifically, the key step is to deform the pair $(X,F)$ to a \HK with a Lagrangian fibration and study the deformed sheaf by restriction to the fibers. 

The examples considered in \cite{ogrady22} are rigid, that is, without infinitesimal deformations. 
O'Grady proved an existence and uniqueness result for stable modular sheaves with certain invariants, which in turn implies a birationality result for the period map for Debarre-Voisin varieties. 

In an effort to construct more examples, Markman \cite{markman21} studied sheaves with obstruction map of rank one.  
The obstruction map for an object $E \in \Db(X)$ is the map
\begin{equation}\label{eq:obstructionMapDef}
    \chi_E : HH^2(X) \to \Ext^2(E,E), \quad \eta \mapsto \eta_E,
\end{equation}
given by evaluation at $E$. 
Here we used that an element $\eta \in HH^2(X)$ in the second Hochschild cohomology group of a smooth projective variety $X$ can be seen as a natural transformation $\id \xrightarrow{\eta} [2]$.   

Similarly, there is a \emph{cohomological} obstruction map
\[
\chi^{\text{coh}}_E : HH^2(X) \to H^*(X,\C),
\]
given by contraction with the Chern character of $E$, see \cite[Definition 6.11]{markman21}. 
Objects with cohomological obstruction map of rank one have been studied independently by Beckmann in \cite{beckmann22}, where they are called \emph{atomic}. 

In \cite[Theorem 1.2]{markman21} it is shown that if $E$ is a torsion-free atomic sheaf on a \HK manifold $X$, then it is modular in the sense above. 
The rank of the (cohomological) obstruction map is invariant under derived equivalences, so this point of view naturally yields an approach to find new modular sheaves: mapping atomic objects to torsion-free sheaves via derived equivalences. 
The most promising class of atomic objects \cite[Theorem 1.8]{beckmann22} consists of line bundles supported on smooth Lagrangians $Z \subset X$, with the property that the restriction map $H^2(X,\C) \to H^2(Z,\C)$ has rank one. 

Besides being modular, atomic sheaves enjoy a crucial extra property.
Namely, on the set of atomic objects one can define an ``extended Mukai vector''.
It lives in the ``extended Mukai lattice'', first introduced in the breakthrough work by Taelman \cite{taelman21}. 
This is the rational vector space
\[
\widetilde{H}(X,\Q) \coloneqq \Q\alpha \oplus H^2(X,\Q) \oplus \Q\beta,
\]
equipped with the quadratic form $\tilde{q}$ obtained by extending the BBF form on $H^2(X,\Q)$ by declaring that $\alpha$ and $\beta$ are orthogonal to $H^2(X,\Q)$, isotropic and $\tilde{q}(\alpha,\beta) = -1$.

The geometric meaning of the classes $\alpha$ and $\beta$ can be understood by the short exact sequence of \cite[Lemma $3.7$]{taelman21}.
Namely, if $\dim(X) = 2n$, we have
\begin{equation}\label{eq:sequenceTaelman}
    0 \to \SH(X) \to \Sym^n \widetilde{H}(X,Q) \to \Sym^{n-2}\widetilde{H}(X,\Q) \to 0,
\end{equation}
where $\SH(X)$ is the Verbitsky component, i.e. the subalgebra of $H^*(X,\Q)$ generated by $H^2(X,\Q)$.
The images of $\alpha^{i}\beta^{n-i}$ under the orthogonal projection $\Sym^{n}\widetilde{H}(X,Q) \to \SH(X)$ generate the monodromy invariant part. 

Taelman \cite[Theorems 2.4, 4.8, 4.9]{taelman21} showed that an equivalence $\Phi : \Db(X) \mor[\sim] \Db(Y)$ induces Hodge isometries $\Phi^{\SH} : \SH(X) \to \SH(Y)$ and $\Phi^{\widetilde{H}} : \widetilde{H}(X,\Q) \to \widetilde{H}(Y,\Q)$.
These two isometries are compatible via the sequence above up to sign, i.e. the diagram 
% https://q.uiver.app/?q=WzAsNCxbMCwwLCJcXFNIKFgpIl0sWzEsMCwiXFxTSChZKSJdLFswLDEsIlxcU3ltXm5cXHRpbGRle0h9KFgsXFxRKSJdLFsxLDEsIlxcU3ltXm5cXHRpbGRle0h9KFksXFxRKSJdLFswLDJdLFsxLDNdLFswLDEsIlxcUGhpXkgiXSxbMiwzLCJcXFN5bV5uXFxQaGlee1xcdGlsZGV7SH19Il1d
\[\begin{tikzcd}
	{\SH(X)} & {\SH(Y)} \\
	{\Sym^n\widetilde{H}(X,\Q)} & {\Sym^n\widetilde{H}(Y,\Q)}
	\arrow[from=1-1, to=2-1]
	\arrow[from=1-2, to=2-2]
	\arrow["{\Phi^{\SH}}", from=1-1, to=1-2]
	\arrow["{\Sym^n\Phi^{\widetilde{H}}}", from=2-1, to=2-2]
\end{tikzcd}\]
commutes up to a sign. 

Building up on this work, Beckmann \cite{beckmann21, beckmann22} and Markman \cite{markman21} introduced the extended Mukai vector $\tilde{v}(E) \in \tilde{H}(X,\Q)$ for any atomic object $E \in \Db(X)$. 
It is only defined up to a constant, by requiring the symmetric power $\tilde{v}^{(n)}$ to be compatible with the usual Mukai vector $v(E) \coloneqq \ch(E)\sqrt{\td_X}$ via the maps in the sequence \eqref{eq:sequenceTaelman}, see Definition \ref{def:defiMukaiVector}. 

The extended Mukai vector inherits some of the properties of the Mukai vector, while at the same time being valued in a smaller, more manageable, vector space than the whole rational cohomology.
Among those, one of the more useful is compatibility with derived equivalences. 
More precisely, if $E \in \Db(X)$ is an atomic object, then 
\[
\Phi^{\widetilde{H}}(\langle \tilde{v}(E) \rangle) = \langle \tilde{v}(\Phi(E)) \rangle,
\]
where $\langle \tilde{v} \rangle$ denotes the line spanned by $\tilde{v}.$

%The Mukai lattice is much more tractable than the entire cohomology of $X$, so this fact is extremely useful to compute Mukai vectors of objects in the derived category of a \HK manifold. 

The investigation of sheaves on \HK manifolds and their moduli spaces is one of the most promising paths to find new examples of \HK manifolds. 
This idea has been around since the works of Kobayashi \cite{kobayashi86} and Verbitsky \cite{verbitsky99}, but the theory is still in its infancy. 
The first step would be to answer the following question posed by Markman. 

\medskip

\noindent
\textbf{Question.} Can we realize OG10 as a moduli space of sheaves on a \HK manifold of type $\mathrm{K3}^{[2]}$?

\medskip

Atomic sheaves have beautiful properties, which make them excellent candidates to have reasonable moduli spaces. 
In this paper, we make progress towards the answer: we find a stable atomic vector bundle on a $\mathrm{K3}^{[2]}$ whose moduli space has an irreducible component birational to OG10.

\subsection{Main results}
The first result of this paper is the construction of a new example of a non-rigid atomic vector bundle. 
Denote by $\sfq_{2i} \in \SH^{4i}(X)$ the classes defined in Definition \ref{def:MonodromyInvariantClasses}, and by $\mathfrak{pt} \in H^{\mathrm{top}}(X,\Q)$ the class of a point. 
Recall that if $X$ is of type $\mathrm{K3}^{[2]}$, then 
\[
c_2(X) = 30 \sfq_2
\]
by \cite[Proposition 2.4]{beckmann_song22}.
We also remand to Section \ref{sec:ProofOfStability} for the notion of $a(\vv)$-generic polarization. 

\begin{Thm}\label{thm:mainTheorem}
Let $X$ be a projective \HK of $\mathrm{K3}^{[2]}$-type.
Consider the Mukai vector
\[
\vv \coloneqq 5\left(1-\frac{3}{4}\sfq_2 + \frac{9}{32}\mathfrak{pt}\right) \in H^*(X,\Q),
\]
and let $h$ be any $a(\vv)$-generic polarization. 
Then, there exists an $h$-stable vector bundle $F_0$ on $X$ with Mukai vector $\vv$. 
Moreover, the group $\Ext^1(F_0,F_0)$ is ten dimensional, the Yoneda pairing is skew-symmetric and induces an isomorphism
\[
\bwed^2 \Ext^1(F_0,F_0) \mor[\sim] \Ext^2(F_0,F_0).
\]
In particular, its deformation functor is smooth.  
\end{Thm}

\noindent
We briefly describe the steps involved in the construction.
\begin{enumerate}
    \item If $X \subset \P^5$ is a general cubic fourfold and $H$ is a general hyperplane, then the structure sheaf $\cO_{F(X \cap H)}$ is an atomic object in $\Db(F(X))$. 
    We degenerate the cubic to the determinantal cubic and consider the corresponding degeneration of the Fano variety of lines. 
    After a resolution, the central fiber is a moduli space $M$ of torsion sheaves on a general K3 surface of degree two, and the surface $F(X \cap H)$ degenerates to a reducible Lagrangian $Z$ with two components. 
    \item The moduli space $M$ is endowed with a Lagrangian fibration $\pi : M \to \P^2$. 
    This Lagrangian fibration has a section, whose image $L$ is one of the components of the reducible Lagrangian $Z$. 
    The other component is a Lagrangian plane $P' \subset M$. 
    As shown in \cite{addington16}, there is an autoequivalence $\Phi$ of $M$ mapping a general point to a line bundle supported on its fiber. 
    We make the following construction: starting from a line bundle $\cL \in \Pic^0(L)$, we glue it with $\cO_{P'}$, to obtain a degree zero line bundle $\overline{\cL}$ on $Z$. 
    The image $\Phi(\overline{\cL})$ is a locally free sheaf, but not slope-stable. 
    \item To make it stable we apply a second autoequivalence: the composition of two (inverses of) $\P$-twists around line bundles.  
    After twisting by a line bundle, the resulting vector bundle will have $c_1 = 0$. 
    Using atomicity we can easily compute the Mukai vector from this construction. 
    Slope-stability, combined with atomicity, allows the bundle to deform to every K\"ahler deformation of $M$ thanks to \cite[Theorem 1.2]{markman21}.
    The Yoneda pairing is studied on the Lagrangian side, by relating it to the cup product on the cohomology of $L$.
\end{enumerate}

Along the way we prove a number of interesting results on their own. We highlight especially the following.

\begin{Prop}[Proposition \ref{prop:locallyFreeness}]
Let $M= M_S(0,H,1-g)$ be a moduli space of torsion sheaves on a general polarized K3 surface $(S,H)$ of genus $g$, and let $\pi: M \to \P^g$ be the Lagrangian fibration. Let $L \subset M$ be a Cohen--Macaulay subvariety such that $\pi|_L : L \to \P^g$ is finite. If $V_L$ is a vector bundle on $L$, then $\Phi(V_L)$ is a locally free sheaf. 
\end{Prop}

The proof is based on an analysis done by Arinkin in \cite{arinkin13} on the singularities of the Fourier-Mukai kernel of $\Phi$. 
This is the first technique to produce \emph{locally free} sheaves starting from push-forwards of locally free sheaves on subvarieties.  
We believe this could be helpful in understanding the relationship between atomic vector bundles and atomic Lagrangians.

The rest of the paper is devoted to the study of the irreducible component of the moduli space $\fM$ of Gieseker semi-stable sheaves on $M$ containing $F_0$. 
While Theorem \ref{thm:mainTheorem} is a general existence result, we study the geometry of $\fM$ only in a particular case.
Namely, $F_0$ is the vector bundle on $M$ obtained from the construction outlined above, and $h$ is a suitable polarization (see Section \ref{sec:ProofOfStability} for a reminder on this notion).
In this context, we are able to prove the following. 

\begin{Thm}[{{Proposition \ref{prop:birationalMap} and Theorem \ref{thm:symplecticForm}}}]\label{thm:introBirational}
The smooth locus $\fM_{\mathrm{sm}}$ is equipped with a closed holomorphic 2-form. Moreover,
there is a birational map preserving the 2-form
\[
X \dashrightarrow \fM,
\]
where $X$ is a \HK manifold of type OG10. 
\end{Thm}

The birational map is easily described. 
Recall that $M$ is a moduli space of sheaves on a general polarized K3 surface $(S,H)$ of degree two. 
The reducible Lagrangian $Z \subset M$ has two components. 
One is the image $P'$ of a section of the Lagrangian fibration $\pi$. 
The other is a Lagrangian surface $L$ isomorphic to $\Sym^2 C$, where $C \subset S$ is a general curve in $|2H|$.

A degree zero line bundle $L_C$ supported on a general curve in $|2H|$ is a general element in the moduli space $M_S(0,2H,-4)$. 
The variety $X$ is the symplectic resolution of this moduli space, and the birational map is given by the steps $(1) - (3)$ applied to the the symmetric square $L_C^{\boxtimes^2}$.

The 2-form at the point $[F] \in \fM_{\mathrm{sm}}$ is given by
\[
\Ext^1(F,F) \times \Ext^1(F,F) \to \C, \quad (a,b) \mapsto \Tr_F\left(\chi_F(\eta) \circ a \circ b\right),
\]
where $\eta \in HH^2(M)$. 
A priori it depends on the choice of $\eta$, but, at least on the image of the birational map, it is unique up to a constant. 
%This is because, as shown in Corollary \ref{cor:obstructionImage}, every sheaf in the image of the birational map above has obstruction map of rank one. 

We conjecture that $\fM$ is itself a \HK manifold of type OG10, in particular that it is smooth. 

A possible way to address this is to analyze the singularities of the sheaves in $\fM$: are they all locally free? Are they all reflexive?
We believe that a positive answer to these questions could lead to an understanding of the singularities of $\fM$. 

\subsection{Structure of the paper}
In Section \ref{sec:Degeneration} we review the works \cite{collino82} and \cite{vandenDries12} to show that there is a smooth family $\cF \to \Delta$ of \HK varieties realizing the degeneration above. 

In Section \ref{sec:NumericalComputations} we review some of the background on atomic sheaves. We also prove new numerical results. First, we compute explicitly the discriminant of an atomic sheaf, reproving that an atomic sheaf is modular.
We use this computation to speculate on Bogomolov's inequality for atomic sheaves.
Then, in Theorem \ref{thm:eulerCharacteristic} we give a formula for the Euler pairing of an atomic object with itself, generalizing the well known formula for K3 surfaces. 

In Section \ref{sec:ConstructionSheaf} we apply construct the bundle of Theorem \ref{thm:mainTheorem}. We show that the the image of a line bundle supported on $Z$ is a locally free sheaf, we compute its Mukai vector and its Ext groups. 

In Section \ref{sec:extGroups} we develop some technical algebraic results which we will need to compute the $\Ext$ groups of the sheaf $F$, and to perform the semistable reduction. The main result of this section is Proposition \ref{prop:ExtExactSequence} where the groups $\Ext^*(\cO_Z,\cO_Z)$ are described in terms of the topology of $Z$. 
%We also study the image under the $\P$-twist $P_{\cE}$ of an object $\cF$, where the pair $(\cE,\cF)$ satisfies the condition $\Ext^*(\cE,\cF) = \C[-1] \oplus \C[-3]$.  

In Section \ref{sec:ssreduction} we perform two inverse $\P$-twists and prove that the resulting object is a stable vector bundle. Locally freeness is proved in Proposition \ref{prop:locallyFreeness}. The key ingredient of the proof of stability is the notion of a suitable polarization introduced in \cite{ogrady21}, which allows to relate stability on the general fiber to global stability. The main use of this is in Proposition \ref{prop:determinant}.

Finally, in Section \ref{sec:moduliSpace} we study the moduli space $\mathfrak{M}$ of Gieseker semistable deformations of $F_0$ on $M$. 
We study the image of the obstruction map in Theorem \ref{thm:symplecticForm}, and prove Theorem \ref{thm:introBirational}.

\noindent
\textbf{Notation and conventions}
Unless otherwise specified, all the functors are derived. 
Where it does not generate confusion, we use the same notation $\cL$ for a line bundle supported on a subvariety and for its pushforward. 
We use O'Grady's normalization for the Fujiki constant: if $X$ is a \HK manifold of dimension $2n$ and $\alpha \in H^2(X)$, then
\[
\int_X{\alpha^{2n}} = c_X \cdot (2n-1)!!\cdot q_X(\alpha)^n.
\] 
where $q_X$ denotes the Beauville-Bogomolov-Fujiki form.

\section{Degenerating the Fano variety of lines}\label{sec:Degeneration}
Let $X_0 \subset \P^5$ be the determinantal cubic, that is the secant variety to the Veronese surface $V \subset \P^5$. It is given in coordinates by
\begin{equation}
    \begin{vmatrix}
x_0 & x_1 & x_2 \\ 
x_2 & x_3 & x_4  \\ 
x_2 & x_4 & x_5   \notag
\end{vmatrix} = 0.
\end{equation}
It is singular along the Veronese surface. If $\P^5$ is identified with the space of conics on a projective plane, $X_0$ corresponds to the singular conics and $V$ to the non-reduced ones. 

Let $X \subset \P^5$ be a very general cubic and let $\cX \to \Delta$ be the pencil spanned by $X_0$ and $X$. If $X = \{ f = 0 \}$, the equation of the pencil is
\begin{equation}
    \begin{vmatrix}
x_0 & x_1 & x_2 \\ 
x_2 & x_3 & x_4  \\ 
x_2 & x_4 & x_5   \notag
\end{vmatrix} + tf = 0.
\end{equation}

Taking the relative Fano variety of lines, we get a family $\cF \to \Delta$ whose general fiber $\cF_t$ is the Fano variety of lines $F(\cX_t)$ of a general member of the pencil. The central fiber $\cF_0 = F(X_0)$ is described in \cite[Propositions $3.2.3$ and $3.2.4$]{vandenDries12}: it is the union of $F_1 \cong (\P^2)^{[2]}$ and $F_2 \cong \P^2 \times (\P^2)^{\vee}$, where $F_1$ is non-reduced with multiplicity four.

\begin{Prop}[{{\cite[Theorem $3.3.7$]{vandenDries12}}}]\label{prop:bigBlowUp}
After a base change along a $2:1$ map $\Delta' \to \Delta$ and blowing up $\cF$ in $F_1$, we get a family $\widehat{\cF} \to \Delta'$ such that
\begin{enumerate}
    \item The special fiber has two irreducible components
    \[
    \widehat{\cF}_0 = E \cup \widehat{F_2}.
    \]
    \item The map $\widehat{F_2} \to F_2$ is an isomorphism, in particular $\widehat{F_2} \cong \P^2 \times (\P^2)^{\vee}$. 
    \item The intersection $E \cap \widehat{F_2} \subset \widehat{F_2}$ is isomorphic to the incidence variety in $\P^2 \times (\P^2)^{\vee}$.
    \item The blow up $\widehat{\cF}$ is smooth along $\widehat{F_2}$.
\end{enumerate}
\end{Prop}

We describe the family $\widehat{\cF}$ in more detail. Since the Veronese surface $V$ has degree two, the intersection $V \cap X$ gives a smooth sextic curve $\Gamma \in \P^2$. Let $p:S \to \P^2$ be the K3 surface obtained as the double cover of $\P^2$ ramified over $\Gamma$.
Let $P \subset S^{[2]}$ be the image of the map 
\[
\P^{2} \to S^{[2]}, \quad x \mapsto p^{-1}(x),
\]
where $p^{-1}(x)$ denotes the schematic fiber. 
Rephrasing \cite[Theorems $3.5.8$ and $3.5.11$]{vandenDries12} gives the following result.

\begin{Thm}
There is a smooth family $\overline{\cF} \to \Delta'$ such that the general fiber $\overline{\cF}_t = F(\cX_t)$ is the Fano variety of lines of the cubic $\cX_t$ and the special fiber $\cF_0$ is isomorphic to $S^{[2]}$. The family $\widehat{\cF}$ is the blow-up of $\overline{\cF}$ in $P$. Under this identification $\widehat{F_2}$ is the exceptional divisor, and $E \cong \mathrm{Bl}_P(S^{[2]})$.
\end{Thm}

Consider the moduli space $M :=M(0,H,-1)$, where $H:= p^*(\cO(1))$. A generic point is represented by a line bundle of degree $0$ supported on a curve in $|H|$. There is a birational map
\begin{equation}\label{eq:flop}
    g: S^{[2]} \dashrightarrow M, \quad \xi \mapsto \omega_C \otimes I_{\xi},
\end{equation}
where $C$ is the unique curve in $|H|$ containing $\xi$. This is well defined outside the plane $P \subset S^{[2]}$. 
The birational map $g$ is the Mukai flop of the plane $P$, and the dual plane $P' \subset M$ is the image of the section of the Lagrangian fibration 
\[
\pi: M \to (\P^2)^{\vee}, \quad  F \mapsto \Supp(F),
\]
where $\Supp(F)$ is the Fitting support. 

\begin{Rem}\label{rem:secondContraction}
Since the cubic $X$ is very general, the plane $P \subset S^{[2]}$ does not deform sideways in $\overline{\cF} \to \Delta'$. The argument in the proof of \cite[Theorem $3.4$]{huybrechts97} shows that the Mukai flop \eqref{eq:flop} can be deformed to $\overline{\cF}$. This implies that $\widehat{\cF}$ can also be contracted to a family $\overline{\cF}' \to \Delta'$ with the same general fiber and special fiber $\overline{\cF}'_0 \cong M$.
\end{Rem}

In \cite{collino82} Collino does the same operations with the Fano variety of lines of a hyperplane section. More precisely, let $H \subset \P^5$ be a general hyperplane. The intersection $V \cap H$ gives a general conic $K \subset \P^2$, and the intersection $X_0 \cap H$ is the secant variety of the image of $K$ via the Veronese embedding. Define $C \coloneqq p^{-1}(K) \subset S$ as the inverse image of the conic via the double cover, it is a genus five curve.

Let $\cX_H \to \Delta$ be the pencil of the hyperplane sections and let $\cZ \subset \cF$ be the relative Fano surface of lines. 
The special fiber is the union of two components $\cZ_0 = Z_1 \cup Z_2 \subset F_1 \cup F_2$, where $Z_2$ is reduced and $Z_1$ is non reduced of multiplicity four. 
Moreover, both $Z_2$ and $Z_{1_{\mathrm{red}}}$ are isomorphic to $\P^2$. 

\begin{Prop}[{{\cite[Proposition $2.1$]{collino82}}}]\label{prop:smallBlowUp}
After a base change along a $2:1$ map $\Delta' \to \Delta$ and blowing up $\cZ$ in $Z_1$, we get a smooth family $\widehat{\cZ} \to \Delta'$ with reducible central fiber
\[
\widehat{\cZ}_0 = E' \cup \widehat{Z_2}. 
\]
Moreover the exceptional divisor $E'$ is isomorphic to $\Sym^2 C$ and $\widehat{Z_2}$ is isomorphic to $Z_2$. 
\end{Prop}

We want to understand the image of $\widehat{\cZ}$ via the contraction $\widehat{\cF} \to \overline{\cF}'$ of Remark \ref{rem:secondContraction}. First observe that the intersection $Z_1 \cap Z_2$ consists of the lines tangent to $K$, so it is isomorphic to $K^*$. Via the embedding $Z_2 \subset \P^2 \times (\P^2)^{\vee}$ it gets mapped into the incidence variety inside $K \times K^*$. In particular it maps isomorphically to its image under both projections. 

Via the contraction $\widehat{\cF} \to \overline{\cF}'$, the component $\widehat{F_2}$ in the central fiber $\widehat{\cF}_0$ gets mapped to $(\P^2)^{\vee}$, so it induces a map $\widehat{Z_2} \to (\P^2)^{\vee}$. This map must be an isomorphism. This is because 
\[
\widehat{Z_2} \cong Z_2 \cong \P^2,
\]
and the contraction maps the intersection $\widehat{Z_2} \cap E'$ isomorphically to its image $K^*$. Hence, the special fiber of $\widehat{\cZ}$ remains unchanged under the contraction $\widehat{\cF} \to \overline{\cF}'$. The same argument also works for the contraction $\widehat{\cF} \to \overline{\cF}$. Summarizing the argument, and adjusting the notation, we showed the following.

\begin{Thm}\label{thm:collino}
There is a smooth family $\cF \to \Delta$ and a smooth subvariety $\cZ \subset \cF$ with the following properties. 
\begin{itemize}
    \item The general fibers $\cF_t$ and $\cZ_t$ are respectively the Fano varieties of lines $F(\cX_t)$ of the cubic $\cX_t$, and of its hyperplane section $F(\cX_t \cap H)$.
    \item The special fiber $\cF_0$ is identified with the moduli space $M = M(0,H,-1)$.
    \item The special fiber $\cZ_0$ is a normal crossing $P' \cup L$, where $L \subset M$ is a Lagrangian surface isomorphic to $\Sym^2 C$. The intersection $L \cap P'$ is isomorphic to $K$. 
\end{itemize}
\end{Thm}

Here, the terminology ``normal crossing'' is used to indicate the union of two smooth varieties which intersect along a smooth divisor. 
We conclude the section with a more detailed description of the geometry of the central fiber. 
If the K3 surface $S$ is very general, and this happens if we choose the cubic $X$ to be very general, the Neron-Severi lattice of the moduli space $M$ is
\begin{equation}\label{eq:NeronSeveriOfModuliSpace}
\NS (M) = \Z \lambda \oplus \Z f, 
\end{equation}
the Beauville-Bogomolov form with respect to this basis has matrix
\[
\begin{pmatrix}
2 & 2 \\
2 & 0 
\end{pmatrix}.
\]
Geometrically, $f:= \pi^*(\cO_{(\P^2)^{\vee}}(1))$ is the inverse image via the Lagrangian fibration of a hyperplane class, and $\lambda$ restricts to a principal polarization on a general fiber $M_t$. From the point of view of the Hilbert scheme, we have
\[
\NS(S^{[2]}) = \Z h \oplus \Z \delta,
\]
where $h$ is the polarization induced by $H=p^*(\cO_{\P^2}(1))$ on $S$, and $\delta$ is half the exceptional divisor of the Hilbert-Chow map. 
The Mukai flop identifies the divisors
\begin{align*}
    h &\longleftrightarrow \lambda, \\
    h - &\delta \longleftrightarrow f.
\end{align*}

\begin{Rem}\label{rem:polarization}
As explained in \cite[Section $3.7$]{vandenDries12} the family $\cF \to \Delta$ of Theorem \ref{thm:collino} is a projective family, and comes equipped with an ample line bundle $\cL$. On the general fiber this line bundle is the Pl\"{u}cker polarization, and on the special fiber $\cF_0 = M$ is  $\cO_M(\lambda + f)$. It has square $6$ and divisibility $2$ on every fiber. 
\end{Rem}

\begin{Prop}\label{prop:geometryofZ2}
The Lagrangian fibration $\pi$ is finite of degree $4$ when restricted to $L$.
\end{Prop}

\begin{proof}
The map $\pi|_{L}: L \to (\P^2)^{\vee}$ is proper, so it suffices to show that it is quasi-finite. 
The intersection $P' \cap L$ is one dimensional and it maps bijectively onto the dual conic $K^*$ via $\pi$. 
On the complement of $P'$ the Mukai flop is an isomorphism, so it suffices show that $\Sym^2 C - K \to (\P^2)^{\vee}$ is quasi-finite. 

The fiber of a line $l \in (\P^2)^{\vee}$ consists of the subschemes $\xi \in S^{[2]}$ mapping to the schematic intersection $l \cap K$. 
The number of such subschemes is always finite. If $l$ intersects $K$ transversely outside the ramification locus $\Gamma$, there are four reduced subschemes mapping to the intersection.
\end{proof}

\section{Numerical computations}\label{sec:NumericalComputations}
\subsection{Review: Atomic and modular sheaves}\label{sec:atomicSheaves}
We briefly review the theory of atomic and modular sheaves. 
The main references are \cite{ogrady21,ogrady22} for modular sheaves and \cite{beckmann21,beckmann22,markman21} for atomic sheaves.  
A more detailed overview of the background is in \cite[Section 2]{beckmann21}.

Let $X$ be a hyper-K\"ahler manifold of dimension $2n$. The \emph{rational extended Mukai lattice} is the rational vector space
\[ 
\widetilde{H}(X,\Q) \coloneqq \Q\alpha \oplus H^2(X,\Q) \oplus \Q\beta.
\]
It is endowed with the non-degenerate quadratic form $\tilde{q}$ obtained by extending the BBF form $q$ on $H^2(X,\Q)$ by declaring that $\alpha$ and $\beta$ are orthogonal to $H^2(X,\Q)$, isotropic and $\tilde{q}(\alpha,\beta) = -1$.
It is also equipped with a Hodge structure, defined by $\widetilde{H}(X,\C)^{2,0} = H^{2,0}(X)$ and imposing compatibility with $\tilde{q}$. 

Let $\SH(X) \subseteq H^*(X,\Q)$ be the \emph{Verbitsky component}, i.e. the subalgebra of $H^*(X,\Q)$ generated by $H^2(X,\Q)$.
By \cite[Lemma $3.7$]{taelman21} there is a short exact sequence 
\begin{equation*}
    0 \to \mathrm{SH}(X) \xrightarrow{\Psi} \Sym^n \widetilde{H}(X,Q) \xrightarrow{\Delta} \Sym^{n-2} \widetilde{H}(X,Q) \to 0.
\end{equation*}
Here $\Delta$ is the Laplacian operator, defined by 
\[
x_1 \cdot \cdot \cdot x_n \mapsto \sum_{i < j}\tilde{q}(x_i,x_j) x_1 \cdot \cdot \cdot \hat{x}_i \cdot \cdot \cdot \hat{x}_j \cdot \cdot \cdot x_n.
\]
The map $\Psi$ is defined as follows. First for every $\lambda \in H^2(X,\Q)$ define the operator $e_{\lambda}$ on $\widetilde{H}(X,\Q)$ by 
\[
e_{\lambda}(\alpha) = \lambda, \quad e_{\lambda}(\beta) =  0, \quad  e_{\lambda}(\mu) = q(\lambda,\mu)\beta \quad \forall \mu \in H^2(X,\Q). 
\]
If we denote by $x^{(n)} \in \Sym^n\widetilde{H}(X,\Q)$ the $n$-th symmetric power of $x \in \widetilde{H}(X,\Q)$, then $\Psi$ is defined as
\[
\lambda_1 \dots \lambda_k \mapsto e_{\lambda_1} \dots e_{\lambda_k}(\alpha^{(n)}/n!),
\]
where $e_{\lambda}$ acts on $\Sym^n \widetilde{H}(X,\Q)$ by derivations. 

Recall that there is an action of the LLV algebra $\fg(X)$ both on $\SH(X)$ and $\widetilde{H}(X,\Q)$, and the injection $\Psi$ is equivariant with respect to this action; for details see \cite{looijengaLunts97, verbitsky96, taelman21}.
Similarly $\Psi$ is an isometry with respect to the bilinear form $b_{\SH}$ on $\SH(X)$ and the one\footnote{To be precise, one needs to rescale by $c_X$ to obtain an isometry, see \cite[Section 2]{beckmann21}.} induced by $\tilde{q}$ on $\Sym^n\widetilde{H}(X,\Q)$.
We denote by
\[
T : \Sym^n \widetilde{H}(X,\Q) \to \mathrm{SH}(X)
\]
the orthogonal projection with respect to the bilinear form on $\Sym^n \widetilde{H}(X,\Q)$ induced by $\tilde{q}$. 

\begin{Defi}[{{\cite[Definition $1.1$]{beckmann22}}}]\label{def:defiMukaiVector}
An object $E \in \Db(X)$ is atomic if there exists a non-zero $\tilde{v}(E) \in \widetilde{H}(X,\Q)$ such that 
\[
\mathrm{Ann}(v(E)) = \mathrm{Ann}(\tilde{v}(E)) \subset \fg(X).
\]
\end{Defi}

\begin{Rem}
    This notion is related to the obstruction map \eqref{eq:obstructionMapDef} as follows. In \cite[Theorem 1.7]{markman21} it is shown that if $E$ is 1-obstructed (that is $\chi_E$ has rank one) and $v(E)$ is not killed by the LLV algebra, then $E$ is atomic. The same result is shown in \cite[Theorem 1.3]{beckmann22} as a consequence of the equivalence between atomicity and cohomological obstruction map of rank one, established in \cite[Theorem 1.2]{beckmann22}. 
\end{Rem} 

\begin{Prop}[{{\cite[Proposition 3.3]{beckmann22} and \cite[Theorem 1.7]{markman21}}}]
If $E \in \Db(X)$ is atomic, the projection $T(\tilde{v}(E)^{(n)})$ is a rational multiple of the projection $v(E)_{\SH}$ of the Mukai vector onto the Verbitsky component. 
\end{Prop}

\subsection{Mukai vector of atomic objects on fourfolds}
Now we consider $X$ a \HK manifold of dimension four.
We want to give an explicit formula for the Mukai vector of an atomic object, in terms of its extended Mukai vector. 
In order to do this, it is necessary to choose a representative for the line spanned by $\tilde{v}(E)$. 
If the rank of $E$ does not vanish, we can normalize $\tilde{v}(E)$ as follows.

\begin{Prop}{{\cite[Theorem $6.13(3)$]{markman21}}}\label{prop:normalization}
Assume $r(E) \neq 0$. Then $\tilde{v}(E)$ can be chosen of the form 
\[
r(E)\alpha + c_1(E) + s(E)\beta,
\]
where $s(E)$ is a rational number. 
\end{Prop}

We are also going to need to understand the projection on the Verbitsky component of certain classes in $\Sym^n\widetilde{H}(X,\Q)$.
For this, we recall the following notation. 

\begin{Defi}[{{\cite[Section 3]{beckmann21}}}]\label{def:MonodromyInvariantClasses}
    Let $X$ be a HK manifold of dimension $2n$. 
    For every $1 \leq i \leq n$, denote by $\mathsf{q}_{2i} \in \SH^{4i}(X,\Q)$ the class defined by the property 
    \begin{equation*}
        \int_X{\omega^{2n-2i}\mathsf{q}_{2i}} = c_X\frac{(2n-2i)!}{2^{n-i}(n-i)!}q(\omega)^{n-i} = c_X(2n-2i-1)!!q(\omega)^{n-i},
    \end{equation*}
    for every $\omega \in H^2(X,\Q)$. 
    For $i = 0$, we set $\sfq_0 \coloneqq 1$.
\end{Defi}

\begin{Rem}
    By \cite[Lemma $2.3$]{beckmann21} the classes $\sfq_{2i}$ generate the monodromy invariant subspace of $\SH^{4i}(X)$ for every $i$.
\end{Rem}

\begin{Lem}{{\cite[Lemma $3.5$]{beckmann21}}}\label{lem:proj_alphabeta}
For $1 \leq i \leq n$ we have 
\[
T(\alpha^{(n-i)}\beta^{(i)}) = (n-i)!\sfq_{2i}.
\]
\end{Lem}

\begin{Lem}\label{lem:projectionTwo}
For every $\gamma \in H^2(X,\Q)$ we have 
\[ 
T(\alpha^{(n-2)} \cdot \gamma^{(2)}) = (n-2)!(\gamma^2 - q(\gamma,\gamma)\sfq_2) \in \SH^4(X).
\]
\end{Lem}

\begin{proof}
By definition 
\[
\Psi(\gamma^2) = e_{\gamma} \cdot e_{\gamma} (\alpha^n/{n!}) = \frac{\alpha^{(n-2)} \cdot \gamma^{(2)}}{(n-2)!} +q(\gamma,\gamma)\frac{\alpha^{(n-1)} \cdot \beta}{(n-1)!}.
\] 
The map $\Psi$ is a section of $T$, so $T(\Psi(\gamma^2)) = \gamma^2.$ Substituting we get 
\begin{align*}
T(\alpha^{(n-2)} \cdot \gamma^{(2)}) &= (n-2)!\left(T(\Psi(\gamma^2)) - q(\gamma,\gamma)\frac{T(\alpha^{(n-1)} \cdot \beta)}{(n-1)!} \right) \\
&= (n-2)!(\gamma^2 - q(\gamma,\gamma)\sfq_2),
\end{align*}
where we used Lemma \ref{lem:proj_alphabeta} in the last equality. 
\end{proof}

\begin{Lem}
Let $X$ be a HK fourfold, and let $\lambda \in H^2(X,\C)$. Then
\[
\int_X{T(\lambda\beta)\mu} = c_Xq(\lambda,\mu),
\]
for every $\mu \in H^2(X,\C)$
\end{Lem}

\begin{proof}
By linearity, we can assume that $q(\lambda,\lambda) \neq 0$. By definition we have
\[
\Psi(\lambda^3) = e_{\lambda} \cdot e_{\lambda} \cdot e_{\lambda} \left(\frac{\alpha^{(2)}}{2}\right) = 3q(\lambda,\lambda)\lambda\beta. 
\]
Using that $\Psi$ is a section of $T$ we obtain $T(\lambda\beta) = \frac{\lambda^3}{3q(\lambda,\lambda)}$, which is easily seen to satisfy the thesis. 
\end{proof}

For any $\lambda \in H^2(X,\Q)$, we denote by $\lambda^{\vee} \in H^6(X,\Q)$ the class such that
\begin{equation}
    \int_X{\lambda^{\vee}\mu} = c_Xq(\lambda,\mu), \text{ for every  }  \mu \in H^2(X,\Q). 
\end{equation}
With this notation, the lemma above says that $T(\lambda\beta) = \lambda^{\vee}$.
Using this, if $\dim(X) = 4$ we can give the explicit form of the Mukai vector of an atomic object. 

\begin{Cor}\label{cor:fullMukaiVector}
Let $X$ be a \HK fourfold, and $E \in \Db(X)$ an atomic object with non-zero rank. Write
\[
\tilde{v}(E) = r\alpha + \lambda + s\beta.
\]
Then we have
\[
v(E)_{\SH} = r + \lambda + \frac{1}{2r}(\lambda^2 - \tilde{q}(\tilde{v}(E),\tilde{v}(E))\sfq_2) + \frac{s}{r}\lambda^{\vee} + \frac{s^2}{2r}\mathsf{q_4}.
\]
\end{Cor}

\begin{proof}
The symmetric square of $\tilde{v}(E)$ is given by
\begin{equation}\label{eq:symmetric_square}
    \tilde{v}(E)^{(2)} = r^2\alpha^{(2)} + 2r\alpha\lambda + 2rs\alpha\beta + \lambda^{(2)} +2s\lambda\beta + s^2\beta^{(2)} \in \Sym^2 \wt{H}(X,\Q).
\end{equation}
Applying the results above, we obtain
\[
T(\tilde{v}(E)^{(2)}) = 2r^2 + 2r\lambda + 2rs\sfq_2 + (\lambda^2 - q(\lambda,\lambda)\sfq_2) + 2s\lambda^{\vee} + s^2\sfq_4. 
\]
Dividing by $2r$ and rearranging the terms we obtain the formula for the Mukai vector in the statement.
\end{proof}

\subsection{Discriminant}\label{sec:Discriminant}
An atomic torsion-free sheaf is modular by the arguments in \cite[Section 5]{beckmann22}. 
Using the computations above, we can give a direct proof of this fact, by computing the discriminant in terms of the extended Mukai vector. 
Bogomolov's inequality will then give an inequality for stable atomic sheaves which is similar to the one for $\mathrm{K3}$ surfaces.  

Let $X$ be a HK manifold of dimension $\dim(X) =2n$, and let $F$ be a torsion-free sheaf on $X$. 
Recall that the discriminant of $F$ is the class
\begin{equation*}\label{eq:discriminant}
    \Delta(F) \coloneqq -2r(F)\ch_2(F) + c_1(F)^2 \in H^4(X,\Q).
\end{equation*}

\begin{Defi}[O'Grady]\label{def:modular}
A torsion-free sheaf $F$ is modular if the projection $\Delta(F)_{\SH}$ on the Verbitsky component is a multiple of the class $\sfq_2$.
\end{Defi}

Define the number $r_X$ as in \cite[Section $3$]{beckmann21}; by Lemma 3.3 in {\it loc. cit.} we have
\begin{equation}\label{eq:DegreeTwoComponentSquareRootTodd}
    (\td_X)^{\frac{1}{2}}_{2,\SH} = r_X\sfq_2.
\end{equation}
Its values for the known deformation types are
\[
r_X=
\begin{cases}
\frac{n+3}{4} & \text{ for $\mathrm{K3^{[n]}}$ or OG10}, \\
\frac{n+1}{4} & \text{ for $\mathrm{Kum_n}$ or OG6}.
\end{cases}
\]

\begin{Prop}\label{prop:discriminant}
Let $F$ be an atomic torsion-free sheaf. Then $F$ is modular, and
\[
 \Delta(F)_{\SH} = (\tilde{q}(\tilde{v}(F),\tilde{v}(F)) + 2r_Xr(F)^2)\sfq_2. 
\]
\end{Prop}

\begin{proof}
Taking the $n$-th symmetric power of $\tilde{v}(F)$ we get
\begin{align*}
  \tilde{v}(F)^{(n)} &= r(F)^{n}\alpha^{(n)} + nr(F)^{n-1}\alpha^{(n-1)}c_1(F) + \binom{n}{2}r(F)^{n-2}\alpha^{(n-2)}c_1(F)^{(2)} \\ &+ nr(F)^{n-1}s(F)\alpha^{(n-1)}\beta + \dots
\end{align*}
Using Lemma \ref{lem:proj_alphabeta} and \ref{lem:projectionTwo} we get
\begin{align*}
    T(\tilde{v}(F)^{(n)}) =& n!r(F)^n + n!r(F)^{n-1}c_1(F) + \binom{n}{2}r(F)^{n-2}(n-2)!(c_1(F)^2  \\
    & - q(c_1(F),c_1(F))\sfq_2) + n!r(F)^{n-1}s(F)\sfq_2 + \dots
\end{align*}
The projection onto $\SH(X)$ of the Mukai vector of $F$ is a rational multiple of this class. Since the rank is non-zero, we deduce that $n!r(F)^{n-1}v(F) = T(\tilde{v}(F)^{(n)})$. Dividing by $n!r(F)^{n-1}$ and comparing the terms of degree four, we get 
\begin{equation}
    v_2(F)_{\SH} = \frac{1}{2r(F)}(c_1(F)^2 -q(c_1(F),c_1(F))\sfq_2) + s(F)\sfq_2.
\end{equation}
On the other hand, by definition $v(F) = \ch(F) \cup (\td_X)^{\frac{1}{2}}$.
By \eqref{eq:DegreeTwoComponentSquareRootTodd} we get
\begin{align*}
    \ch_2(F)_{\SH} & =  v_2(F)_{\SH} - r_Xr(F)\sfq_2  \\
    &= \frac{1}{2r(F)}(c_1(F)^2 - q(c_1(F),c_1(F))\sfq_2) +\left( s(F) - r_Xr(F) \right)\sfq_2.
\end{align*}
Substituting $\ch_2(F)_{\SH}$ in the definition of the discriminant we obtain
\begin{align*}
 \Delta(F)_{\SH} &= (q(c_1(F),c_1(F))+2r_Xr(F)^2-2r(F)s(F))\sfq_2 \\
 & = (\tilde{q}(\tilde{v}(F),\tilde{v}(F)) + 2r_Xr(F)^2)\sfq_2.
\end{align*}
\end{proof}

\begin{Cor}\label{cor:Bogomolov}
If $F$ is an atomic torsion-free slope semistable sheaf, then 
\[
\tilde{q}(\tilde{v}(F),\tilde{v}(F)) + 2r_Xr(F)^2 \geq 0.
\]
\end{Cor}

\begin{proof}
If $F$ is slope semistable for a polarization $H$ on $X$, Bogomolov's inequality gives
\[
\int_X{\Delta(F) \cup H^{n-2}} \geq 0.
\]
The thesis follows from the proposition above because $\int_X{H^{n-2}\sfq_2} = (2n-3)!!q(H)^{n-1} \geq 0$. 
\end{proof}

\begin{Ex}
If $X=S$ is a K3-surface, the inequality
\[ 
v(F)^2 \geq -2r(F)^2
\]
can be improved in the case of a stable sheaf $F$. 
Indeed, in this case it follows from Serre duality and Hirzebruch--Riemann--Roch that 
\[
v(F)^2 \geq -2.
\]
\end{Ex}

\begin{Rem}
It is possible that, similarly to the case of K3 surfaces, a stronger version of the inequality \ref{cor:Bogomolov} could hold. 
Equality should be related to $F$ being a $\P$-object. 
A precise formulation of this inequality seems to be related to understanding how to normalize the extended Mukai vector. 
For example, in \cite[Lemma 4.8 (iii)]{beckmann21} it is shown that if $F$ is in the orbit of the structure sheaf (in particular it is a $\P$-object),
there is a natural normalization for $\tilde{v}(F)$ for which the equality
\[
\tilde{q}(\tilde{v}(F),\tilde{v}(F)) = -2r_X
\]
holds.
\end{Rem}

\subsection{Euler characteristic}
To conclude this section we give a general formula for the Euler pairing of an atomic sheaf with itself, under the assumption that the Mukai vector is contained in the Verbitsky component. Recall that there is a bilinear product $b_{\SH}$ on $\SH(X)$, defined by 
\[
b_{\SH}(\lambda_1 \cdot \dots \lambda_m, \mu_1 \cdot \dots \mu_{2n-m}) \coloneqq (-1)^m \int_X{\lambda_1 \cup \dots \lambda_m \cup \mu_1 \cup \dots \cup \mu_{2n-m}}.
\]

\begin{Lem}
There exists a constant $C$ such that
\[
b_{\SH}\left(T(\tilde{v}^{(n)}),T(\tilde{v}^{(n)})\right) = C\tilde{q}(\tilde{v},\tilde{v})^n.
\]
\end{Lem}

\begin{proof}
Consider the action of the algebraic group $\mathrm{SO}(\widetilde{H}(X,\C))$ on $\SH(X,\C)$ and $\widetilde{H}(X,\C)$, obtained integrating the action of $\fg'_0 \cong \mathfrak{so}(\widetilde{H}(X,\C))$ as explained in \cite[Section 5]{taelman21}.
It acts by isometries with respect to the bilinear forms $b_{\SH}$ and $\tilde{q}$. Moreover, the projection $T : \Sym^n \widetilde{H}(X,\C) \to \SH(X,\C)$ is $\mathrm{SO}(\widetilde{H}(X,\C))$-equivariant, being the projection onto a subrepresentation. Hence, both sides of the equality we want to show are invariant under the action of $\mathrm{SO}(\widetilde{H}(X,\C))$. 

Write $\tilde{v} = r\alpha + \lambda + s\beta$. Dividing by the rank (we can because both sides are homogeneous of degree $n$) and acting by $\exp(e_{\lambda/r})$ we can assume that $\tilde{v} = \alpha + s\beta.$ By definition we have $\tilde{q}(\tilde{v},\tilde{v}) = -2s$.
Moreover, we have 
\[
\tilde{v}^{(n)} = \sum{\binom{n}{i}s^i\alpha^{(n-i)}\beta^{(i)}}.
\]
Applying \ref{lem:proj_alphabeta} we obtain
\[
T(\tilde{v}^{(n)}) = \sum{\frac{n!}{i!}s^i\sfq_{2i}}.
\]
By definition of the Mukai pairing $b_{\SH}$ we get 
\[
    b_{\SH}\left(T(\tilde{v}^{(n)}),T(\tilde{v}^{(n)})\right) =  \left(\sum{\frac{n!}{i!}\frac{n!}{(n-i)!}\int_X{\sfq_{2i}\sfq_{2n-2i}}}\right)s^n 
     = C \tilde{q}(\tilde{v},\tilde{v})^n,
\]
for some constant $C$ independent of $\tilde{v}$. 
\end{proof}

\begin{Thm}\label{thm:eulerCharacteristic}
Let $X$ be a \HK manifold of dimension $\dim(X) = 2n$. 
Let $E \in \Db(X)$ be an atomic object with non-zero rank $r$.
Assume that $v(E)_{\mathrm{SH}} = v(E)$. Then
\[
\chi(E,E) = (-1)^n (n+1)r^2 \left(\frac{\tilde{q}(\tilde{v}(E),\tilde{v}(E))}{2r_Xr^2}\right)^n.
\]
\end{Thm}

\begin{proof}
From the Riemann-Roch Theorem and the assumption that $v(E)_{\SH} = v(E)$ it follows that $\chi(E,E) = b_{\SH}(v(E),v(E))$. Since $T(\tilde{v}^{(n)}) = n!r^{n-1}v(E)$ we obtain from the previous lemma
\[
b_{\SH}(r^{n-1}v(E),r^{n-1}v(E)) = C\tilde{q}(\tilde{v},\tilde{v})^n
\]
for some constant $C$. Dividing both sides by $r^n$ we get 
\begin{equation}\label{eq:homogeneousChi}
    b_{\SH}\left(\frac{v(E)}{r},\frac{v(E)}{r}\right) = C\tilde{q}\left(\frac{\tilde{v}}{r},\frac{\tilde{v}}{r}\right)^n
\end{equation}
To compute the constant $C$, we substitute $\tilde{v}=\alpha + r_X\beta$, the extended Mukai vector of the structure sheaf. Since $r=1$ and $\chi(\cO_X,\cO_X) = n+1$ we get $C= (-1)^n\frac{n+1}{(2r_X)^n}$. 
Substituting $C$ into $\eqref{eq:homogeneousChi}$ and rearranging we get the result.
\end{proof}

\begin{Rem}
The assumption that the Mukai vector is contained in the Verbitsky component is satisfied for every atomic $E$ in the case of \HK varieties of $K3^{[2]}$-type. 
In this case, the formula becomes
\[
\chi(E,E) = 3 \cdot \left(\frac{\tilde{q}(\tilde{v}(E),\tilde{v}(E))}{2rr_X}\right)^2.
\]
Note that it gives a non-trivial integral constraint on the difference $\ext^2(E,E) - 2\ext^1(E,E)$. 
Finding an independent restriction on its possible values, for example in the form of a bound on $\ext^2(E,E)$, could be a path to investigate smoothness of the moduli space of semistable deformations of $E$. 
\end{Rem}

\section{Homological algebra of normal crossings Lagrangians}\label{sec:extGroups}
In this section we develop some homological algebra aimed towards the computation of the Ext groups $\Ext^k(\cO_Z,\cO_Z)$, where $Z=L \cup P' \subset M$ is the central fiber of the family of Theorem \ref{thm:collino}. 
The results in this section will be important both for computing the Ext groups $\Ext^k(F,F)$, and to study stability of the sheaf $F$ in Section \ref{sec:ssreduction}.

\subsection{\P-twist}\label{sec:P-twists}
We begin with some computations which we will prove useful later, especially in Section \ref{sec:ssreduction} to perform the semistable reduction.
Let $X$ be a \HK fourfold, and let $\cE$ and $\cF$ be two coherent sheaves on $X$. 
We make the following assumptions
\begin{enumerate}
    \item $\cE$ is a $\P$-object, that is $\Ext^*(\cE,\cE)$ is isomorphic as an algebra to $H^*(\P^2,\C)$.
    \item $\Ext^*(\cE,\cF) \cong \C[-1] \oplus \C[-3]$, and it is non-trivial as a module over $\Ext^*(\cE,\cE)$.
\end{enumerate}
In particular, there is a unique non-trivial extension
\[
0 \to \cF \to \cG \to \cE \to 0.
\]

To a $\P$-object $\cE$ one can associate an autoequivalence $P_{\mathcal{E}}$ of $\Db(X)$ called the $\P$-twist around $\cE$. 
Here we briefly recall the definition, for details see \cite[Section $2$]{huybrechts_thomas06}.
Let $h \in \Ext^2(\cE,\cE)$ a generator.
Define the map $\overline{h}^{\vee} : \Ext^{*-2}(\cE,\cF) \to \Ext^{*}(\cE,\cF)$ as the precomposition with $h$. 
The $\P$-twist around $\cE$ applied to $\cF$ can be described as
\begin{equation}
    P_{\cE}(\cF) = C\left(C(\Ext^{*-2}(\cE,\cF) \otimes \cE \xrightarrow{\overline{h}^{\vee}\cdot \id - \id \cdot h} \Ext^{*}(\cE,\cF) \otimes \cE) \to \cF \right).
\end{equation}
Here we used the notation $C(A \rightarrow B)$, to indicate the cone of the morphism $A \to B \in \Db(X)$. 

\begin{Rem}\label{rem:octahedral}
By the octahedral axiom one can see that $P_{\cE}(\cF)$ can be equivalently described as the cone of the map
\[
\Ext^*(\cE,\cF) \otimes \cE[-1] \to C(\Ext^*(\cE,\cF) \otimes \cE \to \cF).
\]
\end{Rem}

We want to compute the cohomology sheaves of the complex $P_{\cE}(\cF)$. We first compute the ones of the cone of the evaluation map. 

\begin{Lem}\label{lem:cohomologyEV}
Consider the evaluation map 
\[
\Ext^*(\cE,\cF) \otimes \cE \to \cF.
\]
The cohomology sheaves of its cone $C$ are 
\[
\cH^{k}(C) \cong
\begin{cases}
\cG & \text{for } k=0,\\ 
\cE & \text{for } k=2,\\
0 & \text{otherwise.}
\end{cases}
\]
\end{Lem}

\begin{proof}
The long exact sequence in cohomology gives the two sequences
\begin{align*}
    0 &\to \cF \to \cH^0(C) \to \cE \to 0, \\
    0 &\to \cH^{2}(C) \to \cE \to 0.
\end{align*}
The rest of the long exact sequence shows that there is no cohomology in degrees different from $0$ and $2$. 
The first sequence is induced by the evaluation map $\Ext^1(\cE,\cF) \otimes \cE \to \cF$. 
Therefore it is not split, and $\cH^{0}(C) \cong \cG$. 
\end{proof}

\begin{Prop}\label{prop:cohomology_Ptwist}
The cohomology sheaves of $P_{\cE}(\cF)$ are given by
\[
\cH^{k}(P_{\cE}(\cF)) \cong
\begin{cases}
\cG & \text{for } k=0,\\
\cE & \text{for } k=3.
\end{cases}
\]
In particular, there is a distinguished triangle
\[
\cG \to P_{\cE}(\cF) \to \cE[-3].
\]
\end{Prop}

\begin{proof}
Consider the distinguished triangle 
\[
\Ext^*(\cE,\cF) \otimes \cE[-1] \to C(\Ext^*(\cE,\cF) \otimes \cE \to \cF) \to P_{\cE}(\cF)
\]
of Remark \ref{rem:octahedral}. Applying the long exact sequence of cohomology sheaves and Lemma \ref{lem:cohomologyEV} we get the exact sequences
\begin{align*}
0 &\to \cG \to \cH^0(P_{\cE}(\cF)) \to 0,\\
0 &\to \cH^1(P_{\cE}(\cF)) \to \cE \to \cE \to \cH^2(P_{\cE}(\cF)) \to 0, \\ 
0 &\to \cH^3(P_{\cE}(\cF)) \to \cE \to 0.
\end{align*}
If we check that the middle map $\cE \to \cE$ in the second sequence is the identity we are done. By definition it is induced in $\cH^2$ by the map
\[ 
\Ext^*(\cE,\cF) \otimes \cE[-1] \to C,
\]
which in turn is obtained from the octahedral axiom, composed with the isomorphism in Lemma \ref{lem:cohomologyEV}. Chasing the definitions and the commutativity in the octahedral axiom one sees that the desired map is induced in $\cH^2$ by the map
\[
H[-1] : \Ext^*(\cE,\cF) \otimes \cE[-1] \to \Ext^*(\cE,\cF) \otimes \cE[1],
\]
described explicitly as 
% https://q.uiver.app/?q=WzAsNSxbMCwwLCJcXGNFWy0yXSJdLFsxLDAsIlxcY0UiXSxbMCwxLCJcXGNFWy00XSJdLFsxLDEsIlxcY0VbLTJdIl0sWzAsM10sWzAsMSwiaCJdLFsyLDMsImhbLTJdIl0sWzAsMywiXFxtYXRocm17aWR9Il0sWzAsMiwiXFxiaWdvcGx1cyIsMix7InN0eWxlIjp7ImJvZHkiOnsibmFtZSI6Im5vbmUifSwiaGVhZCI6eyJuYW1lIjoibm9uZSJ9fX1dLFsxLDMsIlxcYmlnb3BsdXMiLDIseyJzdHlsZSI6eyJib2R5Ijp7Im5hbWUiOiJub25lIn0sImhlYWQiOnsibmFtZSI6Im5vbmUifX19XV0=
\[\begin{tikzcd}
	{\cE[-2]} & \cE \\
	{\cE[-4]} & {\cE[-2]} 
	\arrow["h", from=1-1, to=1-2]
	\arrow["{h[-2]}", from=2-1, to=2-2]
	\arrow["{\mathrm{id}}", from=1-1, to=2-2]
	\arrow["\oplus"', draw=none, from=1-1, to=2-1]
	\arrow["\oplus"', draw=none, from=1-2, to=2-2]
\end{tikzcd}\] 
in the proof of \cite[Lemma 2.5]{huybrechts_thomas06}. 
From this description it is clear that the induced map in $\cH^2$ is the identity. 
\end{proof}

\begin{Cor}\label{cor:ssReduction}
The object $P^{-1}_{\cE}(\cG)$ sits in a distinguished triangle 
\[
\cE \to P^{-1}_{\cE}(\cG) \to \cF.
\]
\end{Cor}

\begin{proof}
From \cite[Lemma $2.5$]{huybrechts_thomas06} we see that $P_{\cE}(\cE) \cong \cE[-4]$. Applying the equivalence $P^{-1}_{\cE}$ to the distinguished triangle 
\[
\cG \to P_{\cE}(\cF) \to \cE[-3]
\]
of Proposition \ref{prop:cohomology_Ptwist} we obtain
\[
P^{-1}_{\cE}(\cG) \to \cF \to \cE[1].
\]
Rotating this triangle gives the thesis. 
\end{proof}

\begin{Cor}\label{cor:vanishings}
If $\cE$ and $\cF$ are as above, we have 
\[
\Ext^{k}(\cE,\cG) \cong
\begin{cases}
0 & \text{if } k \neq 4, \\
\C & \text{if } k = 4.
\end{cases}
\]
\end{Cor}

\begin{proof}
Setting $\cG' := P_{\cE}^{-1}(\cG)$ we get 
\begin{align*}
    \Ext^{k}(\cE,\cG) &= \Ext^k(\cE,P_{\cE}(\cG')) = \Ext^k(\cE[4],\cG') \\ & = \Ext^{k-4}(\cE,\cG').
\end{align*}
Both objects $\cG$ and $\cG'$ are sheaves, so the ext groups above vanish for $k \neq 4$. For $k=4$, the exact sequence
\[
\Ext^4(\cE,\cE) \to \Ext^4(\cE,\cG) \to \Ext^4(\cE,\cF) = 0,
\]
shows that it is at most one dimensional. It is non-zero, because of the map $\cG \to \cE$, so it is one dimensional. 
\end{proof}

\subsection{Normal crossings Lagrangians}
We consider a normal crossings Lagrangian subvariety 
\[
Z = Z_1 \cup Z_2 \subset X
\]
in a \HK variety of dimension $2n$. 
That is $Z_1$ and $Z_2$ are smooth Lagrangians, and their scheme theoretic intersection $W := Z_1 \cap Z_2$ is smooth of dimension $n-1$; in particular 
\[
T_{Z_1}|_W \cap T_{Z_2}|_W = T_W.
\]

\begin{Rem}\label{Rem:sympdual}
Let $\sigma_X$ denote the symplectic form on $X$. Since $T_W \subset T_{Z_i}|_W$ and $Z_i$ is Lagrangian, we have
\[
\sigma_X(v,w) = 0 \text{ for every } v \in T_{Z_i} \text{ and } w \in T_W.
\]
The sum $T_{Z_1}|_W + T_{Z_2}|_W$ is a subbundle of $T_X|_W$ of rank $n+1$, so it is the symplectic orthogonal to $T_W$. 
\end{Rem}

The following result was shared by E. Markman through personal communication with the author.

\begin{Lem}[Markman]\label{lem:normalsAreDual}
The normal bundle $N_{W/Z_1}$ is dual to $N_{W/Z_2}$.  
\end{Lem}

\begin{proof}
Consider the following diagram
% https://q.uiver.app/?q=WzAsMTcsWzEsMSwiVF9XIl0sWzEsMiwiVF9XIFxcb3BsdXMgVF9XIl0sWzEsMywiVF9XIl0sWzIsMSwiVF97Wl8xfXxfVyBcXGNhcCBUX3taXzJ9fF9XIl0sWzIsMiwiVF97Wl8xfXxfVyBcXG9wbHVzIFRfe1pfMn18X1ciXSxbMiwzLCJUX3taXzF9fF9XICsgVF97Wl8yfXxfVyJdLFszLDIsIk5fe1cvWl8xfSBcXG9wbHVzIE5fe1cvWl8yfSJdLFszLDMsIihUX3taXzF9fF9XICsgVF97Wl8yfXxfVykvVF9XIl0sWzEsMCwiMCJdLFsxLDQsIjAiXSxbMiwwLCIwIl0sWzIsNCwiMCJdLFswLDIsIjAiXSxbMCwzLCIwIl0sWzQsMywiMCJdLFszLDRdLFs0LDIsIjAiXSxbMCwxXSxbMSwyXSxbMiw1XSxbMSw0XSxbMCwzLCJcXGNvbmciXSxbMyw0XSxbNCw1XSxbNSw3XSxbNiw3LCJcXGNvbmciXSxbNCw2XSxbOCwwXSxbMiw5XSxbNSwxMV0sWzEwLDNdLFsxMiwxXSxbMTMsMl0sWzcsMTRdLFs2LDE2XV0=
\[\begin{tikzcd}
	& 0 & 0 \\
	& {T_W} & {T_{Z_1}|_W \cap T_{Z_2}|_W} \\
	0 & {T_W \oplus T_W} & {T_{Z_1}|_W \oplus T_{Z_2}|_W} & {N_{W/Z_1} \oplus N_{W/Z_2}} & 0 \\
	0 & {T_W} & {T_{Z_1}|_W + T_{Z_2}|_W} & {(T_{Z_1}|_W + T_{Z_2}|_W)/T_W} & 0 \\
	& 0 & 0 & {}
	\arrow[from=2-2, to=3-2]
	\arrow[from=3-2, to=4-2]
	\arrow[from=4-2, to=4-3]
	\arrow[from=3-2, to=3-3]
	\arrow["\cong", from=2-2, to=2-3]
	\arrow[from=2-3, to=3-3]
	\arrow[from=3-3, to=4-3]
	\arrow[from=4-3, to=4-4]
	\arrow["\cong", from=3-4, to=4-4]
	\arrow[from=3-3, to=3-4]
	\arrow[from=1-2, to=2-2]
	\arrow[from=4-2, to=5-2]
	\arrow[from=4-3, to=5-3]
	\arrow[from=1-3, to=2-3]
	\arrow[from=3-1, to=3-2]
	\arrow[from=4-1, to=4-2]
	\arrow[from=4-4, to=4-5]
	\arrow[from=3-4, to=3-5]
\end{tikzcd}\]
The nine lemma implies that the right vertical map is an isomorphism. 
From the previous remark we see that 
\[
(T_{Z_1}|_W + T_{Z_2}|_W)/T_W \cong (T_W)^{\perp}/T_W,
\]
which is a symplectic rank two bundle, in particular it has trivial determinant. We conclude that
\[
N_{W/Z_1} \otimes N_{W/Z_2} \cong \bigwedge\nolimits^2(N_{W/Z_1} \oplus N_{W/Z_2}) \cong \cO_X.
\]
\end{proof}

We define the vector bundle
\[
\tilde{N}:=T_X|_W/(T_{Z_1}|_W + T_{Z_2}|_W),
\]
following the notation in the Appendix of \cite{cks_2003}.
A diagram chase gives the following.

\begin{Lem}\label{lem:symplecticity}
There is an isomorphism of short exact sequences 
% https://q.uiver.app/?q=WzAsMTAsWzEsMCwiTl97Vy9aXzF9Il0sWzIsMCwiTl97Wl8yL1h9fF9XIl0sWzMsMCwiXFx0aWxkZXtOfSJdLFsxLDEsIk5ee1xcdmVlfV97Vy9aXzJ9Il0sWzIsMSwiXFxPbWVnYV97Wl8yfXxfVyJdLFszLDEsIlxcT21lZ2FfVyJdLFswLDEsIjAiXSxbNCwxLCIwIl0sWzAsMCwiMCJdLFs0LDAsIjAiXSxbNiwzXSxbMyw0XSxbNCw1XSxbNSw3XSxbOCwwXSxbMCwxXSxbMSwyXSxbMiw5XSxbMCwzLCJcXGNvbmciXSxbMSw0LCJcXGNvbmciXSxbMiw1LCJcXGNvbmciXV0=
\[\begin{tikzcd}
	0 & {N_{W/Z_1}} & {N_{Z_2/X}|_W} & {\tilde{N}} & 0 \\
	0 & {N^{\vee}_{W/Z_2}} & {\Omega_{Z_2}|_W} & {\Omega_W} & 0
	\arrow[from=2-1, to=2-2]
	\arrow[from=2-2, to=2-3]
	\arrow[from=2-3, to=2-4]
	\arrow[from=2-4, to=2-5]
	\arrow[from=1-1, to=1-2]
	\arrow[from=1-2, to=1-3]
	\arrow[from=1-3, to=1-4]
	\arrow[from=1-4, to=1-5]
	\arrow["\cong", from=1-2, to=2-2]
	\arrow["\cong", from=1-3, to=2-3]
	\arrow["\cong", from=1-4, to=2-4]
\end{tikzcd}\]
\end{Lem}

\begin{proof}
The first short exact sequence is given by
\[
0 \to (T_{Z_1}|_W + T_{Z_2}|_W)/T_{Z_2}|_W \to T_X|_W/(T_{Z_2}|_W) \to T_X|_W/(T_{Z_1}|_W + T_{Z_2}|_W) \to 0,
\]
and noting that $(T_{Z_1}|_W + T_{Z_2}|_W)/T_{Z_2}|_W \cong T_{Z_1}|_W / T_W$.
The central vertical map in the diagram is induced by the restriction of the isomorphism $\sigma_X : T_X \cong \Omega_X$. 
The composition
\[
T_{Z_1} + T_{Z_2} \to T_X \cong \Omega_{X} \to \Omega_W
\]
vanishes by Remark \ref{Rem:sympdual}. 
So the central map factors to give the diagram in the statement. 
\end{proof}

Denote by $j_i : Z_i \hookrightarrow X$ the embeddings. If $E_1$ and $E_2$ are locally free sheaves on $Z_1$ and $Z_2$ we can compute the $\mathrm{Ext}$ groups $\Ext^{k}(j_{1,*}E_1,j_{2,*}E_2)$ using the following spectral sequence. 

\begin{Thm}[{{\cite[Theorem A.1]{cks_2003}}}]\label{Thm:spectralSequence}
With the above notation, there is a convergent spectral sequence 
\[
E^{p,q}_2 := H^p(W,E_1^{\vee}|_W \otimes E_2|_W \otimes N_{W/Z_2} \otimes \bigwedge \nolimits^{q-1}\tilde{N}) \implies \Ext^{p+q}(j_{1,*}E_1,j_{2,*}E_2).
\]
\end{Thm}

\begin{Ex}\label{Ex:mixedexts}%
If $X$ has dimension $4$, and $E= \cO_{Z_1}$ and $F= \cO_{Z_2}(-W)$, then by Lemma \ref{lem:symplecticity} we have
\[
E_2^{p,q} = H^p(N_{W/Z_2}^{\vee} \otimes N_{W/Z_2} \otimes \Omega^{q-1}_W) = H^p(\Omega^{q-1}_W).
\]
The spectral sequence degenerates at the $E_2$ page by degree reasons, giving 
\[
\Ext^k(\cO_{Z_1},\cO_{Z_2}(-W)) \cong H^{k-1}(W,\C).
\]
\end{Ex}

\begin{Prop}\label{prop:ExtExactSequence}
Assume that $X$ has dimension four. Then, there is a long exact sequence 
\[
H^k(Z_2,\C) \to \Ext^k(\cO_{Z_2}(-W),\cO_Z) \to H^{k-1}(W,\C) \to H^{k+1}(Z_2,\C),
\]
where the connecting homomorphism is the pushforward in cohomology along the inclusion $W \subset Z_2$.
In particular, there is an isomorphism
\[
\Ext^k(\cO_{Z_2}(-W),\cO_Z) \cong H^k(Z_2 - W,\C).
\]
\end{Prop}

\begin{proof}
Consider the long exact sequence obtained applying $\Hom(\cO_{Z_2}(-W),-)$ to
\[
0 \to \cO_{Z_2}(-W) \to \cO_Z \to \cO_{Z_1} \to 0.
\]
Since $Z_2 \subset X$ is a Lagrangian surface, by dimensional reasons the local-to-global spectral sequence vanishes and yields
\begin{equation}\label{eq:extLagrangianPure}
    \Ext^k(\cO_{Z_2}(-W),\cO_{Z_2}(-W)) \cong H^k(Z_2,\C).
\end{equation}
Example \ref{Ex:mixedexts} implies that
\begin{equation}\label{eq:extLagrangianMixed}
    \Ext^k(\cO_{Z_2}(-W),\cO_{Z_1}) \cong H^{k-1}(W,\C).
\end{equation}
Therefore, we only need to show that the connecting homomorphisms 
\[
\Ext^k(\cO_{Z_2}(-W),\cO_{Z_1}) \to \Ext^{k+1}(\cO_{Z_2}(-W),\cO_{Z_2}(-W))
\]
become identified with the pushforwards in cohomology. The Serre dual statement is that the connecting map 
\[
\Ext^k(\cO_{Z_2}(-W),\cO_{Z_2}(-W)) \to \Ext^{k+1}(\cO_{Z_1},\cO_{Z_2}(-W))
\]
is the restriction $H^k(Z_2,\C) \to H^k(W,\C)$. 

The isomorphisms \eqref{eq:extLagrangianPure} and \eqref{eq:extLagrangianMixed} are induced by the degeneration of the spectral sequences:
\[
H^p\RHom(\cH^{-q}(j_2^*j_{2,*}\cO_{Z_2}(-W)),\cO_{Z_2}(-W)) \implies \Ext^{p+q}(\cO_{Z_2}(-W),\cO_{Z_2}(-W))
\] and 
\[
H^p\RHom(\cH^{-q}(j_2^*j_{1,*}\cO_{Z_1}),\cO_{Z_2}(-W)) \implies \Ext^{p+q+1}(\cO_{Z_1},\cO_{Z_2}(-W)).
\]
The connecting homomorphism is induced by pullback along the map 
\[
\cO_{Z_1} \to \cO_{Z_2}(-W)[1].
\]
Taking $j_2^*$ and $\cH^{-q}$ we get the zero map in cohomology for every $q$. This implies that the long exact cohomology sequence induced by $j_2^*\cO_{Z_1} \to j_2^*cO_{Z_2}(-W)[1]$ is actually a collection of short exact sequences, represented by maps 
\begin{equation}\label{eq:mapsCoh}
  \cH^{-q}(j_2^*j_{1,*}\cO_{Z_1}) \to \cH^{-q}(j_2^*j_{2,*}\cO_{Z_2}(-W))[1].  
\end{equation}
Pulling back along those maps gives a map on the $E_2$ page of the spectral sequences, which induces the connecting homomorphism that we wish to understand.

Using \cite[Proposition $A.6$]{cks_2003} we obtain $\cH^{-q}(j_2^*j_{1,*}\cO_{Z_1}) \cong i_{2,*}\bigwedge\nolimits^q\tilde{N}^{\vee}$, where $i_k : W \to Z_k$ is the inclusion. The map \eqref{eq:mapsCoh} becomes a map 
\[
i_{2,*}\bigwedge\nolimits^q\tilde{N}^{\vee} \to \bigwedge\nolimits^q N_{Z_2/X}^{\vee} \otimes \cO_{Z_2}(-W)[1].
\]
Verdier duality gives that $i^!_2 = i_2^* \otimes \cO_W(W)[-1]$, so the map becomes $\bigwedge\nolimits^q\tilde{N}^{\vee} \to \bigwedge\nolimits^q N_{Z_2/X}^{\vee}$, which is identified with the restriction map via Lemma \ref{lem:symplecticity}. 

The isomorphism 
\[
\Ext^k(\cO_{Z_2}(-W),\cO_Z) \cong H^k(Z_2 - W,\C),
\]
follows from the five Lemma applied to the long exact sequence obtained by Poincar\'{e} and Lefschetz duality. 
\end{proof}

\section{Construction of the bundle}\label{sec:ConstructionSheaf}
Let $M = M_S(0,H,-1)$ be the moduli space appearing as the central fiber in the family of Theorem \ref{thm:collino}.
As we recall below, there exists an autoequivalence
\[
\Phi : \Db(M) \isom \Db(M)
\]
whose kernel is the relative Poincaré sheaf.
Let $Z  = P' \cup L \subset M$ the central fiber of the family of Lagrangians of Theorem \ref{thm:collino}.

We make the following construction.
Start with $\cL$ a line bundle of degree zero on $L$. 
Since the intersection $ P' \cap L = K^*$ is rational, the restriction $\cL|_{K^*}$ is trivial. 
Hence $\cL$ glues with the structure sheaf of $P'$ and gives a global line bundle $\overline{\cL}$ on $Z$. 
This means that we have a short exact sequence 
\begin{equation}\label{eq:ShortExactSequenceGluing}
    0 \to \cL(-K^*) \to \overline{\cL} \to \cO_{P'} \to 0
\end{equation}
of sheaves on $X$.
The goal of this section is to study the image $\Phi( \overline{\cL})$, and precisely to prove the following. 

\begin{Prop}\label{prop:MukaiVector}
Let $\Phi : \Db(M) \isom \Db(M)$ and $Z$ be as above, and define 
\[
F \coloneqq \Phi( \overline{\cL}) \in \Db(M).
\]
\begin{enumerate}
    \item The object $F$ is a locally free sheaf of rank five with extended Mukai vector
    \[
    \tilde{v}(F) =  5\left(\alpha + 3f - \frac{3}{4}\beta\right). 
    \]
    \item Defining $F_0 \coloneqq F \otimes \cO(-3f)$, we have 
\[
\tilde{v}(F_0) = 5\left(\alpha-\frac{3}{4}\beta\right)
\]
and 
\[
v(F_0) = 5\left(1-\frac{3}{4}\sfq_2 + \frac{9}{32}\mathfrak{pt}\right).
\]
\end{enumerate}
\end{Prop}

\subsection{Relative Poincar\'{e} sheaf}
Let $(S,H)$ be a polarized K3 surface such that every curve in $|H|$ is integral, e.g. $(S,H)$ general. The moduli space $M:=M_{S}(0,H,d)$ is equipped with a Lagrangian fibration 
\[
M \to |H|
\]
realizing it as the relative compactified Jacobian $\Pic^{d+g-1}(\cC/|H|$) of the the universal curve over the linear system $|H|$. 
In particular, a general point is a line bundle of degree $d$ on a smooth curve in the linear system $|H|$.

In \cite{addington16} the authors extend to the relative compactified Jacobian the construction of the Poincar\'{e} sheaf done by Arinkin \cite{arinkin13} for the Jacobian of singular (integral) curves. In the case of $d = 1 - g$, we obtain a sheaf 
\[
\cP \in \Coh(M \times_{|H|} M).
\]
Taking the Fourier-Mukai transform we obtain a functor
\[
\Phi \coloneqq \Phi_{\cP} : \Db(M) \to \Db(M),
\]
which is an autoequivalence by \cite[Theorem C]{arinkin13}. 
By construction, $\Phi$ maps a general point $x \in M$ to a line bundle over the Jacobian of $C = \pi(x)$. 
The sheaf $\cP$ is only defined up to a normalization, which we fix by requiring that 
\begin{equation}\label{eq:normalizationPoincaré}
   \Phi(\cO_{P'}) = \cO_M.
\end{equation}

The most important aspect (for us) of the autoequivalence $\Phi$ is its ability to turn Cohen--Macaulay sheaves into vector bundles. 

\begin{Prop}\label{prop:locallyFreeness}
Let $M= M_S(0,H,1-g)$ be a moduli space of torsion sheaves on a general polarized K3 surface $(S,H)$ of genus $g$, and let $\pi: M \to \P^g$ be the Lagrangian fibration. Let $L \subset M$ be a subvariety such that $\pi|_L : L \to \P^g$ is finite. If $V_L$ is a Cohen--Macaulay sheaf on $L$, then $\Phi(V_L)$ is a locally free sheaf. 
\end{Prop}

\begin{proof}[Proof of Proposition \ref{prop:locallyFreeness}]
Since the Poincar\'{e} sheaf $\cP \in \Coh(M \times_{\P^g} M)$ is flat with respect to both projections by \cite[Theorem A]{arinkin13}, $\pi_1^{*}(\cM_{L}) \otimes \cP$ is a sheaf on $M \times_{\P^g} M$. So, $\Phi(V_L)$ is a complex concentrated in non negative degrees. To show that it is a locally free sheaf, it suffices to prove
\[
\Ext_M^i(\Phi(V_L),\C(x)) = 0 \quad \text{for } i > 0,
\]
for every $x \in M$, where $\C(x)$ denotes the skyscraper sheaf at $x$. From \cite[Proposition 7.1]{arinkin13} it follows that $\Phi^{-1}(\C(x)) = \cP_{M \times \{x^{\vee}\}}[g]$, where $x^{\vee}$ parameterizes the dual sheaf to $x$. So, we have 
\[
\Ext_M^i(\Phi(V_L),\C(x)) = \Ext_M^i(V_L,\cL_t[g]),
\]
where $t:=\pi(x)$, and $\cL_t := \cP_{M \times \{x^{\vee}\}}$ is a torsion-free rank one sheaf supported on $M_t$. 
Since $V_L$ is a Cohen--Macaulay sheaf of dimension $g$ on $M$, the derived dual $\RcHom_M(V_L,\cO_M)$ is just $\mathcal{E}xt^g(V_L,\cO_M)[-g]$. 
Hence we have
\begin{align*}
    \Ext_M^i(V_L,\cL_t[g]) & \isom \H^i(M, \RcHom_M(V_L,\cO_M) \otimes^{L} \cL_t[g]) \\
    & \isom \H^i(M, \mathcal{E}xt^g(V_L,\cO_M) \otimes^{L} \cL_t).
\end{align*}
The sheaf
\[
\mathcal{H}^i(\mathcal{E}xt^g(V_L,\cO_M) \otimes^{L} \cL_t) = \mathcal{T}or_k^{\cO_M}(\mathcal{E}xt^g(V_L,\cO_M),\cL_t)
\]
vanishes by \cite[Corollary to Theorem V.4]{serre65}.
Indeed, by \cite[Theorem A(2)]{arinkin13} the sheaf $\cL_t$ is Cohen--Macaulay of dimension $g$ on $M$, and the same holds for $\mathcal{E}xt^g(V_L,\cO_M)$. 
So we have
\[
\H^i(M, \mathcal{E}xt^g(V_L,\cO_M) \otimes^{L} \cL_t)= H^i(M, \mathcal{E}xt^g(V_L,\cO_M) \otimes \cL_t) = 0 \quad \text{for } i > 0,
\]
because the sheaf $\mathcal{E}xt^g(V_L,\cO_M) \otimes \cL_t$ is supported on $M_t \cap L$ which is finite. 
\end{proof}

We conclude our overview of the autoequivalence $\Phi$ with the following Lemma, which allows us to understand the restriction to a general fiber. 

\begin{Lem}\label{Lem:restrictingPhi}
Let $M = M_S(0,H,1-g)$ where $(S,H)$ is a general polarized K3 of genus $g$.
For every $E \in \Db (M)$ we have 
\[
\Phi_{\cP}(E)|_{M_t} = i_{M_t,*}\Phi_{\cP_t}(E|_{M_t})  
\]
for $t \in (\P^n)^{\vee}$ a general point. 
\end{Lem}

\begin{proof}
First note that the equivalence $\Phi$ is $(\P^n)^{\vee}$-linear, because the kernel $\cP$ is defined on the fiber product. In particular there is an isomorphism of functors 
\[
\Phi(-) \otimes \cO_{M_t} = \Phi(- \otimes \cO_{M_t}).
\]
Projection formula gives the isomorphism of functors $-\otimes \cO_{M_t} \cong i_{M_t,*}i_{M_t}^*(-)$. To conclude it remains to prove that 
\[
\Phi(i_{M_t,*}(-)) \cong i_{M_t,*}\Phi_{\cP_t}(-),
\]
which follows from the base change Theorem as explained in \cite[Lemma $11.30$]{huybrechts2006}. 
\end{proof}

\subsection{Computation of the class}\label{sec:computationClass}
We are almost ready to prove Proposition \ref{prop:MukaiVector}, the last remaining piece is to compute the exteded Mukai vector of the structure sheaf $\cO_Z$. 
As explained in \cite[Section $8.1$]{beckmann22}, the structure sheaf of $F(X_t \cap H)$ is an atomic sheaf on $F(X_t)$. 
Using \cite[Lemma $7.3$]{markman21} we can determine the line generated by its Mukai vector in $\widetilde{H}(X,Q)$, which is
\begin{equation}\label{eq:MukaiLineFanoSurface}
   \langle \tilde{v}(\cO_{F(X_t \cap H)}) \rangle = \langle h_t - 3\beta \rangle,
\end{equation}
where $h_t$ is the Pl\"ucker polarization on $F(X_t)$. 

\begin{Lem}\label{lem:MukaiLineLagrangians}
The structure sheaves $\cO_{P'}$ and $\cO_{P' \cup L}$ are atomic sheaves on $M$ with Mukai lines spanned respectively by
\[
  \langle \tilde{v}(\cO_{P'}) \rangle = \langle \lambda -3f + 3\beta \rangle
\]
and
\[
 \langle \tilde{v}(\cO_{P' \cup L}) \rangle = \langle \lambda +f - 3\beta \rangle
\]
\end{Lem}

\begin{proof}
The statement for $P'$ follows from \cite[Lemma $7.3$]{markman21}. The statement for the union $P' \cup L$ follows from \eqref{eq:MukaiLineFanoSurface} combined with Theorem \ref{thm:collino} and Remark \ref{rem:polarization}.
\end{proof}

\begin{proof}[Proof of Proposition \ref{prop:MukaiVector}.]
Consider the short exact sequence \eqref{eq:ShortExactSequenceGluing}
\[
0 \to \cL(-K^*) \to \overline{\cL} \to \cO_{P'} \to 0
\]
After applying the autoequivalence $\Phi$ this becomes 
\[
0 \to \Phi(\cL(-K^*)) \to F \to \cO_ M \to 0.
\]
Here $\Phi(\cL(-K^*))$ is locally free by Proposition \ref{prop:locallyFreeness}, and has rank four by Lemma \ref{Lem:restrictingPhi}. 
It follows that $F$ is locally free of rank five. 

Now we compute its extended Mukai vector.
Notice that since $\cL$ has degree zero, we have $v(\cL) = v(\cO_L)$, so it suffices to compute the extended Mukai vector of $\Phi(\cO_Z)$. 
We start by describing the action of the equivalence $\Phi$ on $\widetilde{H}(X,\Q)$. 
We follow the computations done in \cite[Proposition $10.4$]{beckmann21} for odd genus. 

First, since the skyscraper sheaf of a point goes to a line bundle of degree $0$ on a fiber, we deduce that $\beta \mapsto f$. From the autoduality property of the Poincar\'{e} sheaf, described in \cite[Equation $(7.8)$]{arinkin13}, we see that $f \mapsto \beta$. 
The choice \eqref{eq:normalizationPoincaré} for the normalization of $\cP$ implies
\[
\lambda - 3f +3\beta \mapsto -2\left(\alpha + \frac{5}{4}\beta\right),
\]
where the coefficient $-2$ is determined by imposing the map $\Phi^{\widetilde{H}}$ to be an isometry. 
This implies that
\begin{align*}
    \lambda + f -3\beta = \lambda - 3f + 3\beta + (4f -6\beta) &\mapsto  -2 \left(\alpha + \frac{5}{4}\beta\right) + (4\beta - 6f) \\ 
    &= -2\alpha -6f +\frac{3}{2}\beta.
\end{align*}
Since $\tilde{v}(\cO_Z) =  \lambda + f -3\beta$ by Lemma \ref{lem:MukaiLineLagrangians}, the formula for the extended Mukai vector follows by \cite[Theorem 1.7(4)]{markman21}. 
Since by Lemma \ref{Lem:restrictingPhi} $F$ has rank five, its normalized extended Mukai vector is 
\[
\tilde{v}(F) = 5\alpha +15f - \frac{15}{4}\beta.
\]
Twisting by $\cO(-3f)$ we kill the first Chern class, and the extended Mukai vector becomes
\[
\tilde{v}(F_0) = 5\left(\alpha-\frac{3}{4}\beta\right).
\]
This can be computed, for example, using the fact that tensor product with a line bundle induces an isometry on $\wt{H}(X,\Q)$.
Corollary \ref{cor:fullMukaiVector} then gives
\begin{equation*}
v(F_0) = 5\left(1-\frac{3}{4}\sfq_2 + \frac{9}{32}\mathfrak{pt}\right),
\end{equation*}
because if $c_X =1$, then $\sfq_4 = \mathfrak{pt}$. 
\end{proof}

\subsection{Ext groups}
To conclude this section, we apply the results of Section \ref{sec:extGroups} to compute the Ext groups $\Ext^*(F,F)$. 
Being obtained from $\overline{\cL} \in \Db(M)$ through an equivalence, it suffices to compute the Ext groups of $\overline{\cL}$.

\begin{Lem}\label{lem:extComputation}
With the above notation we have isomorphisms
\begin{align*}
    &\Ext^1(\overline{\cL},\overline{\cL}) \cong H^1(L,\C) \\
    &\Ext^2(\overline{\cL},\overline{\cL})  \cong \mathrm{Cok}\left(H^0(K^*,\C) \to H^2(L,\C)\right)
\end{align*}
\end{Lem}

\begin{proof}
Again, we start with the short exact sequence \eqref{eq:ShortExactSequenceGluing}
\[
0 \to \cL(-K^*) \to \overline{\cL} \to \cO_{P'} \to 0.
\]
The isomorphisms \eqref{eq:extLagrangianPure} and \eqref{eq:extLagrangianMixed} remain valid for the same reasons. 
The rest of the proof of Proposition \ref{prop:ExtExactSequence} is not affected by the twist, so it remains to show that
\begin{equation}\label{eq:isomorphismsLongExact}
    \Ext^i(\overline{\cL},\overline{\cL}) \cong \Ext^i(\cL(-K^*),\overline{\cL}) \text{ for } i =1,2.
\end{equation}
The sheaves $\cE = \cO_{P'}$ and $\cF = \cL(-K^*)$ satisfy the assumptions of Section \ref{sec:P-twists} by Theorem \ref{Thm:spectralSequence}. 
Thus \eqref{eq:isomorphismsLongExact} is a consequence of the vanishings in Corollary \ref{cor:vanishings}. 
\end{proof}

\begin{Cor}\label{cor:YonedaSkewSymmetric}
 The Yoneda pairing is skew-symmetric and induces an isomorphism
\begin{equation*}
    \bwed^2 \Ext^1\left(\overline{\cL},\overline{\cL}\right) \xrightarrow{\sim} \Ext^2\left(\overline{\cL},\overline{\cL}\right), \quad a \wedge b \to a \circ b.
\end{equation*}
\end{Cor}

\begin{proof}
    This relies on the fact that the Lagrangian $L$ is in fact the symmetric square of a genus five curve $C$. 
    In fact we have  
\begin{equation*}
 H^2(L,\C) \cong H^2(C,\C) \oplus \bigwedge\nolimits^2 H^1(L,\C),
\end{equation*}
where the second summand is embedded via cup product.
The fundamental class of $K^* \subset \Sym^2 C$ spans the direct summand $H^2(C,\C)$.
Thus, taking the cokernel as in Lemma \ref{lem:extComputation} we get
\begin{equation}\label{eq:Ext2wedge}
    \Ext^2\left(\overline{\cL},\overline{\cL}\right) \cong \bigwedge\nolimits^2 H^1(L,\C).
\end{equation}
To conclude, recall that by \cite[Theorem 2.1.5]{mladenov19}, the isomorphism 
\[
\Ext^*(\cL,\cL) \cong H^*(L,\C)
\]
preserves the algebra structure. 
Hence, the cokernel map
\[
H^2(L,\C) \to \Ext^2\left(\overline{\cL},\overline{\cL}\right)
\]
maps the cup product to the Yoneda product.
The statement then follows from \eqref{eq:Ext2wedge}.
\end{proof}

\section{Semistable reduction}\label{sec:ssreduction}
In this section, we assume fixed a line bundle $\cL \in \Pic^0(L)$, and we examine the stability of the bundle $F$ constructed in Section \ref{sec:ConstructionSheaf}.
In particular, we will show that it is not stable, and to obtain a slope-stable vector bundle we apply two (inverses) $\P$-twists. 

By our normalization of the Poincar\'{e} sheaf we have $\Phi(\cO_{P'}) = \cO_M$, and Propositions \ref{prop:locallyFreeness} and  \ref{prop:geometryofZ2} imply that $G \coloneqq \Phi(\cL(-K^*))$ is a vector bundle of rank four. 
Therefore we have a short exact sequence
\begin{equation}\label{eq:sequenceF}
    0 \to G \to F \to \cO_M \to 0,
\end{equation}
obtained applying $\Phi$ to \eqref{eq:ShortExactSequenceGluing}.

As noted in the proof of Lemma \ref{lem:extComputation}, the sheaves $\cE = \cO_{P'}$ and $\cF = \cO_{L}(-K^*)$ are as in the setting of Section \ref{sec:P-twists}. 
By Corollary \ref{cor:ssReduction} the inverse $\P$-twist of $\overline{\cL}$ around $\cO_{P'}$ lives in a short exact sequence 
\[
0 \to \cO_{P'} \to P^{-1}_{\cO_{P'}}(\overline{\cL}) \to \cL(-K^*) \to 0.
\]
and we set 
\begin{equation}
    F' := \Phi\left(P^{-1}_{\cO_{P'}}\left(\overline{\cL}\right)\right) \in \Db(M).
\end{equation}
By construction of the $\P$-twist we have $F' \cong P^{-1}_{\cO_M}(F)$.
Applying the equivalence $\Phi$ we get a short exact sequence
\begin{equation*}\label{eq:FasExtension}
    0 \to \cO_M \to F' \to G \to 0.
\end{equation*}
In particular, the sheaf  $F'$ is locally free of rank five.

\begin{Lem}
The vector bundles $F$ and $F'$ are unstable for any polarization $h$ on $M$. 
\end{Lem}

\begin{proof}
By \cite[Remark 2.4]{huybrechts_thomas06} any $\P$-twist acts as the identity in cohomology, so $F$ and $F'$ have the same (extended) Mukai vector, which can be computed using Proposition \ref{prop:MukaiVector}.
In particular
\[
\rk(F) = \rk(F') = 5 \text{ and } c_1(F) = c_1(F')= 15f.
\]
The slope with respect to a polarization $h$ is  
\[
\mu(F)=\mu(F') = 3q(h,f) >0.
\] 
So the sequence \eqref{eq:sequenceF} destabilizes $F$.

To destabilize $F'$, first recall that the normal bundles of $K^*$ in $P'$ and $L$ are dual to each other
\[
\cO_L(K^*)|_{K^*} \cong \cO_{P'}(2)|_{K^*}^{\vee},
\]
as we proved in Lemma \ref{lem:normalsAreDual}.
Since the restriction $\cL|_{K^*}$ is trivial, because $K^*$ is rational, we 
\[
\Hom(\cL(-K^*),\cO_{P'}(2)|_{K^*}) = H^0(K^*,\cO_{K^*}) = \C.
\]
The unique map
\[
\cL(-K^*) \twoheadrightarrow \cO_{P'}(2)|_{K^*},
\]
must be a twisting of the canonical map associated to the embedding $K^* \subset L$, in particular is surjective.
Since $\Ext^1(P_{\cO_{P'}}^{-1}(\overline{\cL}),\cO_{P'}) = 0$, we can lift the composite map
\[
P_{\cO_{P'}}^{-1}(\cO_Z) \to \cL(-K^*) \to \cO_{P'}(2)|_{K^*},
\]
to a diagram
% https://q.uiver.app/?q=WzAsMTAsWzIsMCwiUF97XFxjT197UCd9fV57LTF9KFxcY09fWikiXSxbMywwLCJcXGNPX3tMfSgtS14qKSJdLFsxLDAsIlxcY09fe1AnfSJdLFs0LDAsIjAiXSxbMCwwLCIwIl0sWzMsMSwiXFxjT197UCd9KDIpfF97S14qfSJdLFsyLDEsIlxcY09fe1AnfSgyKSJdLFsxLDEsIlxcY09fe1AnfSJdLFswLDEsIjAiXSxbNCwxLCIwIl0sWzIsMF0sWzAsMV0sWzEsM10sWzQsMl0sWzgsN10sWzcsNl0sWzYsNV0sWzUsOV0sWzAsNiwiIiwwLHsic3R5bGUiOnsiaGVhZCI6eyJuYW1lIjoiZXBpIn19fV0sWzEsNSwiIiwwLHsic3R5bGUiOnsiaGVhZCI6eyJuYW1lIjoiZXBpIn19fV0sWzIsNywiIiwxLHsibGV2ZWwiOjIsInN0eWxlIjp7ImhlYWQiOnsibmFtZSI6Im5vbmUifX19XV0=
\begin{equation}\label{diagram:destabilizingF'}
\begin{tikzcd}
	0 & {\cO_{P'}} & {P_{\cO_{P'}}^{-1}\left(\overline{\cL}\right)} & {\cL(-K^*)} & 0 \\
	0 & {\cO_{P'}} & {\cO_{P'}(2)} & {\cO_{P'}(2)|_{K^*}} & 0
	\arrow[from=1-2, to=1-3]
	\arrow[from=1-3, to=1-4]
	\arrow[from=1-4, to=1-5]
	\arrow[from=1-1, to=1-2]
	\arrow[from=2-1, to=2-2]
	\arrow[from=2-2, to=2-3]
	\arrow[from=2-3, to=2-4]
	\arrow[from=2-4, to=2-5]
	\arrow[two heads, from=1-3, to=2-3]
	\arrow[two heads, from=1-4, to=2-4]
	\arrow[Rightarrow, no head, from=1-2, to=2-2]
\end{tikzcd}
\end{equation}
where the short exact sequence below is the defining sequence of the inclusion $K^* \subset P'$.
When applying $\Phi$, the vertical central map becomes a non-zero morphism $F' \to \cO_M(2f)$.
The inequality
\[
\mu(\cO_M(2f)) = 2q(h,f) < 3q(h,f) = \mu(F')
\]
show that $F'$ is destabilized by this map. 
\end{proof}

To obtain a stable bundle $F''$, we replicate the construction of $F'$ with an additional $\P$-twist around the line bundle $\cO_M(2f)$.
Namely, define
\begin{equation*}
    F'' \coloneqq P_{\cO_M(2f)}^{-1}(F'),
\end{equation*}
and notice that 
\[
F'' \isom \Phi\left(P_{\cO_{P'}(2)}^{-1}(P^{-1}(\overline{\cL}))\right)
\]
by construction. 
A diagram chase in \eqref{diagram:destabilizingF'} shows that 
\[
\Ker\left(P_{\cO_{P'}}^{-1}\left(\overline{\cL}\right) \to \cO_{P'}(2)\right) = \cL(-2K^*).
\]
So, defining $G' := \Phi(\cO_L(-2K^*))$ we have a short exact sequence
\[
0 \to G' \to F' \to \cO_M(2f) \to 0.
\]
From the spectral sequence in Proposition \ref{prop:ExtExactSequence} we see that the pair $\cE = \cO_M(2f)$ and $\cF=\cL(-2K^*)$ satisfies the assumptions of Section \ref{sec:P-twists} 
Via the equivalence $\Phi$, Corollary \ref{cor:ssReduction} provides a short exact sequence
\begin{equation}\label{eq:sequenceF''}
    0 \to \cO_M(2f) \to F'' \to G' \to 0,
\end{equation}
from which we deduce that $F''$ is a locally free sheaf of rank five. 
Notice that $\mu(\cO_M(2f)) < \mu(F'')$, so this sequence does not destabilize.  

\begin{Rem}\label{rem:atomicity}
%Since $Z \subset M$ is a degeneration of the Fano surface of lines inside the Fano variety of lines, its structure sheaf $\cO_Z$ is an atomic object. 
The bundles $F,F'$ and $F''$ are all atomic, because they are obtained from $\cO_Z$ by derived equivalences. 
They all have the same Mukai vector, because the $\P$-twist acts as the identity in cohomology. 
\end{Rem}

\subsection{Proof of stability}\label{sec:ProofOfStability}
Our next goal is to show that $F''$ is slope-stable for some polarization $h$. 
Since it is modular by Proposition \ref{prop:discriminant}, we can use the results in \cite{ogrady22} for slope-stability for modular sheaves. 

We are interested in slope-stability for suitable polarizations (see \cite[Definition $3.5$]{ogrady22}).
Intuitively, a polarization $h$ is suitable if it is very close to the nef divisor $f$.
More precisely, as shown in \cite[Section 3]{ogrady22}, for any modular sheaf $F$ on a HK manifold $X$ there is a wall and chamber decomposition of the ample cone of $X$.

This decomposition depends only on the number 
 \[
 a(F) \coloneqq \frac{r(F)^2 \cdot d(F)}{4c_X},
\]
where $d(F) \in \Q$ is defined by the equality 
\[
    \Delta(F)_{\SH} = d(F)\sfq_2.
\]
We say that a polarization is $a(F)$-generic if it belongs in one of the chambers. 
When we want to highlight that $a(F)$ depends only on the Mukai vector $\vv=v(F)$ and not on the sheaf itself, we will say that a polarization is $a(\vv)$-generic. 

At least in the case of a projective \HK of Picard rank two, a polarization is suitable if it lives in the chamber (of the ample cone) whose closure contains $f$. 
Stability with respect to an $a(F)$-suitable polarization is special, because it allows to study stability of $F$ by stability of the restriction to a general fiber, see \cite[Section 3.5]{ogrady21}

\begin{Rem}\label{rem:ComputationOfa(F)}
    Using the computations in Section \ref{sec:computationClass} and Section \ref{sec:Discriminant} we can compute the number $a(F'')$. 
    By Proposition \ref{prop:discriminant} we have 
    \[
    d(F'') = 25\tilde{q}\left( \alpha - \frac{3}{4}\beta \right) + \frac{5}{2}\cdot 25 = 100,
    \]
    and thus 
    \[
    a(F'') = \frac{25 \cdot 100}{4} = 625. 
    \]
\end{Rem}

In what follows by a suitable polarization we mean a $625$-suitable polarization in the sense of \cite[Definition $3.5$]{ogrady21}.

\begin{Prop}\label{prop:determinant}
Let $h$ be a suitable polarization on $M$. 
If $E \subset F''$ is $h$-destabilizing, then $c_1(E) = b \cdot f$ for some $b \in \Z$. 
\end{Prop}

\begin{proof}
By \cite[Proposition $3.4$]{ogrady21} it is enough to show the statement for a rational ample class in the same chamber as $h$, i.e. we can assume $h = f + \varepsilon\lambda$ for $0 < \varepsilon << 1$. 
We write
\[
c_1(E) = bf + c\lambda,
\]
with respect to the decomposition of \eqref{eq:NeronSeveriOfModuliSpace}.
Let $t \in (\P^2)^{\vee}$ be a general point. 
Lemma \ref{Lem:restrictingPhi} together with the exact sequence \eqref{eq:sequenceF''}, implies that for a general $t$ 
\[
F''_t = \cO_{M_t} \oplus L_{t,1} \oplus \dots \oplus L_{t,4},
\]
where $L_i$ are line bundles of degree zero on $M_t$.  
Therefore, the restriction $F''_t$ is semistable because it is the sum of line bundles of degree $0$.
So we have
\[
    2c\varepsilon = \int_{M_t}{c_1(E_t) \cup h_t} \leq \int_{M_t}{c_1(F''_t) \cup h_t} = 0,
\]
which gives $c \leq 0$.

By definition we have 
\[
\mu (E) = \frac{1}{\rk(E)}\int_{M}{(bf+c\lambda)\cup h^3},  
\]
where $h^3 = 3\varepsilon(f^2 \cup \lambda) + 3\varepsilon^2(f \cup \lambda^2) + \varepsilon^3 \lambda^3$, because $f^3 = 0$. The class $f^2$ is Poincar\'{e} dual to a general fiber, so we have 
\[
\int_M{f^2 \cup \lambda^2} = \int_{M_t}{\lambda_t^2} >0.
\]
In particular, in $\mu (E)$ there is a term in $\varepsilon$, namely $\mu(E) = 3c\varepsilon\int_M{f^2 \cup \lambda^2} + \varepsilon^2(\dots)$. On the other hand, since $c_1(F'') = 15f$ and $\rk(F'') = 5$, we have $\mu(F'') = 9\varepsilon^2(\dots).$ The assumption that $h$ is destabilizing gives the inequality 
\[
\mu(E) \geq \mu(F'') \quad \forall \varepsilon <<1.
\]
Passing to the limit $\varepsilon \to 0$, we obtain that the term in $\varepsilon$ must be non-negative, i.e. $c \geq 0$. Combining with the previous inequality we get $c = 0$.
\end{proof}

Assume that there exists $A \subset F''$ a destabilizing subsheaf. 
By definition $0 < \rk(A) < \rk(F'')$, and we can assume that $A$ is saturated in $F''$, that is $B \coloneqq F''/A$ is torsion-free. 

\begin{Lem}\label{lem:dichotomyRank}
With the above notation, either $\rk(A) = 1$ or $\rk(A) = 4$. 
\end{Lem}

\begin{proof}
On a general fiber $M_t$ we can write
\[
F''_t = L_{t,0} \oplus L_{t,1} \oplus \dots \oplus L_{t,4},
\]
where $L_{t,0} = \cO_{M_t}$, and $L_{t,i}$ are non-trivial line bundles of degree zero. 
The restriction $A_t$ has the same slope as $F''_t$, hence it is a sub-sum of these line bundles,
\[
A_t = L_{t,i_1} \oplus \dots \oplus L_{t,i_r}.
\]
Taking $\Phi^{-1}$, Lemma \ref{Lem:restrictingPhi} gives
\[
\Phi^{-1}(A)|_{M_t} = i_{M_t,*}\Phi_{\cP_t}^{-1}(A_t) =\cO_{M_t,[L_{t,i_1}]} \oplus \dots \oplus \cO_{M_t,[L_{t,i_r}]}.
\]
We deduce that over an open $U \subset \P^2$, the support $\Supp \Phi^{-1}(A) \subseteq Z$ and it is finite over the base of degree $r$. Since $r < 5$ by assumption $\Supp \Phi^{-1}(A)$ it not equal to the whole $Z$. 
Hence it must be one of the two components, giving the dichotomy in the statement.  
\end{proof}

\begin{Thm}
The bundle $F''$ is slope-stable with respect to any suitable polarization $h$. 
\end{Thm}
\begin{proof}
Recall that $\rk(F'')=5$ and $c_1(F'')=15f$. Let 
\[
0 \to A \to F'' \to B \to 0,
\]
be a slope destabilizing short exact sequence. 
We assume $A$ saturated, so $B$ is torsion-free. 
By Lemma \ref{lem:dichotomyRank} either $\rk(A) = 1$ or $\rk(A) = 4$. 

\medskip
\noindent
\textbf{Case 1.} Assume that $\rk(A) = 1$. 
Consider the following commutative diagram with exact rows and columns
% https://q.uiver.app/?q=WzAsMjEsWzEsMSwiXFxjT19NIFxcY2FwIEEiXSxbMiwxLCJBIl0sWzMsMSwiQS8oXFxjT19NIFxcY2FwIEEpIl0sWzEsMiwiXFxjT19NIl0sWzIsMiwiRiJdLFszLDIsIkciXSxbMSwzLCJcXGNPX00vKFxcY09fTSBcXGNhcCBBKSJdLFsyLDMsIkIiXSxbMywzLCJCLyhcXGNPX00vKFxcY09fTSBcXGNhcCBBKSkiXSxbMSw0LCIwIl0sWzMsMCwiMCJdLFszLDQsIjAiXSxbMiw0LCIwIl0sWzIsMCwiMCJdLFsxLDAsIjAiXSxbMCwxLCIwIl0sWzAsMiwiMCJdLFswLDMsIjAiXSxbNCwxLCIwIl0sWzQsMiwiMCJdLFs0LDMsIjAiXSxbMCwzXSxbMyw2XSxbNiw5XSxbMCwxXSxbMiw1XSxbMyw0XSxbNCw1XSxbNiw3XSxbNyw4XSxbNSw4XSxbMSw0XSxbNCw3XSxbMSwyXSxbMTAsMl0sWzgsMTFdLFs3LDEyXSxbMTMsMV0sWzE0LDBdLFsxNSwwXSxbMTYsM10sWzE3LDZdLFsyLDE4XSxbNSwxOV0sWzgsMjBdXQ==
% https://q.uiver.app/?q=WzAsMTgsWzEsMSwiXFxjT19NKDJmKSBcXGNhcCBBIl0sWzIsMSwiQSJdLFszLDEsIkEvKFxcY09fTSgyZikgXFxjYXAgQSkiXSxbMSwyLCJcXGNPX00oMmYpIl0sWzIsMiwiRicnIl0sWzMsMiwiRyciXSxbMSwzLCJcXGNPX00oMmYpLyhcXGNPX00oMmYpIFxcY2FwIEEpIl0sWzIsMywiQiJdLFsxLDQsIjAiXSxbMywwLCIwIl0sWzIsNCwiMCJdLFsyLDAsIjAiXSxbMSwwLCIwIl0sWzAsMSwiMCJdLFswLDIsIjAiXSxbMCwzLCIwIl0sWzQsMSwiMCJdLFs0LDIsIjAiXSxbMCwzXSxbMyw2XSxbNiw4XSxbMCwxXSxbMiw1XSxbMyw0XSxbNCw1XSxbNiw3XSxbMSw0XSxbNCw3XSxbMSwyXSxbOSwyXSxbNywxMF0sWzExLDFdLFsxMiwwXSxbMTMsMF0sWzE0LDNdLFsxNSw2XSxbMiwxNl0sWzUsMTddXQ==
\[\begin{tikzcd}
	& 0 & 0 & 0 \\
	0 & {\cO_M(2f) \cap A} & A & {A/(\cO_M(2f) \cap A)} & 0 \\
	0 & {\cO_M(2f)} & {F''} & {G'} & 0 \\
	0 & {\cO_M(2f)/(\cO_M(2f) \cap A)} & B \\
	& 0 & 0
	\arrow[from=2-2, to=3-2]
	\arrow[from=3-2, to=4-2]
	\arrow[from=4-2, to=5-2]
	\arrow[from=2-2, to=2-3]
	\arrow[from=2-4, to=3-4]
	\arrow[from=3-2, to=3-3]
	\arrow[from=3-3, to=3-4]
	\arrow[from=4-2, to=4-3]
	\arrow[from=2-3, to=3-3]
	\arrow[from=3-3, to=4-3]
	\arrow[from=2-3, to=2-4]
	\arrow[from=1-4, to=2-4]
	\arrow[from=4-3, to=5-3]
	\arrow[from=1-3, to=2-3]
	\arrow[from=1-2, to=2-2]
	\arrow[from=2-1, to=2-2]
	\arrow[from=3-1, to=3-2]
	\arrow[from=4-1, to=4-2]
	\arrow[from=2-4, to=2-5]
	\arrow[from=3-4, to=3-5]
\end{tikzcd}\]
The intersection $A \cap \cO(2f)$ is defined as the kernel of the map
\[
A \oplus \cO(2f) \to F'', \ (a,x) \mapsto a-x.
\]
Restricting to the general fiber, both $A$ and $\cO(2f)$ become trivial, so wee see that the intersection is non-trivial, because $F''|_t$ has only one trivial summand. 
Since $A$ has rank one, the quotient sheaf $A/(\cO_M(2f) \cap A)$ has rank $0$.
It and embeds into $G'$, which is locally free, so it is zero, which gives
\[
A \subset \cO_M(2f).
\]
Using that $B$ is torsion-free, the same argument yields $\cO_M(2f) \subset A$. 
We deduce that $A = \cO_M(2f)$, which is not destabilizing. 

\medskip
\noindent
\textbf{Case 2.} Assume that $\rk(A) = 4$. The quotient $B$ is a torsion-free rank one sheaf, so it injects into its double dual $B^{\vee \vee}$, which is a line bundle on $M$ by \cite[Proposition 1.11]{hartshorne80}. 
By Proposition \ref{prop:determinant}, it suffices to show that $\Hom(F'',\cO(kf))=0$ vanishes for every $k \leq 3$. 
By construction we have 
\[
\Hom(F'',\cO(2f)) = \Hom(P^{-1}_{\cO_{M}(2f)}(F'),\cO_{M}(2f)) = 0.
\]
It follows using $h^0(M,\cO_M(f)) \neq 0$ that $\Hom(F'',\cO_M(kf))= 0$ for every $k \leq 2$. Hence it remains to show that $\Hom(F'',\cO_M(3f)) = 0$. 

Let $\varphi: F'' \to \cO_M(3f)$ be a morphism, and let $D$ be its kernel.  
Restricting to a general fiber $M_t$ we see that $\Hom(\cO_M(2f),D) = 0$, because $D_t$ splits as a sum of four non trivial line bundles of degree $0$. 
This implies that the composition
\[
\cO_M(2f) \to F'' \to \cO_M(3f)
\]
is not zero, hence it is injective. 
Applying $\Phi^{-1}$ we obtain a diagram
% https://q.uiver.app/?q=WzAsMTAsWzAsMCwiMCJdLFsxLDAsIlxcY09fe1AnfSgyKSJdLFsyLDAsIlBeey0xfV97XFxjT197UCd9KDIpfShQXnstMX1fe1xcY09fe1AnfX0oXFxjT19aKSkiXSxbMywwLCJcXGNPX3tMfSgtMkteKikiXSxbMSwxLCJcXGNPX3tQJ30oMikiXSxbMiwxLCJcXGNPX3tQJ30oMykiXSxbMywxLCJcXGNPX3tsfSgzKSJdLFs0LDAsIjAiXSxbNCwxLCIwIl0sWzAsMSwiMCJdLFsxLDQsIiIsMCx7ImxldmVsIjoyLCJzdHlsZSI6eyJoZWFkIjp7Im5hbWUiOiJub25lIn19fV0sWzAsMV0sWzEsMl0sWzIsM10sWzMsN10sWzQsNV0sWzUsNl0sWzYsOF0sWzksNF0sWzIsNSwiXFxQaGleey0xfShcXHZhcnBoaSkiXSxbMyw2XV0=
\[\begin{tikzcd}
	0 & {\cO_{P'}(2)} & {P^{-1}_{\cO_{P'}(2)}\left(P^{-1}_{\cO_{P'}}\left(\overline{\cL}\right)\right)} & {\cL(-2K^*)} & 0 \\
	0 & {\cO_{P'}(2)} & {\cO_{P'}(3)} & {\cO_{l}(3)} & 0
	\arrow[Rightarrow, no head, from=1-2, to=2-2]
	\arrow[from=1-1, to=1-2]
	\arrow[from=1-2, to=1-3]
	\arrow[from=1-3, to=1-4]
	\arrow[from=1-4, to=1-5]
	\arrow[from=2-2, to=2-3]
	\arrow[from=2-3, to=2-4]
	\arrow[from=2-4, to=2-5]
	\arrow[from=2-1, to=2-2]
	\arrow["{\Phi^{-1}(\varphi)}", from=1-3, to=2-3]
	\arrow[from=1-4, to=2-4]
\end{tikzcd}\]
where $\cO_{l}$ is the structure sheaf of a line $l \subset P'$, i.e. of the zero locus of a section of $\cO_{P'}(1)$. Since $l \not \subseteq L$, the map $\cL(-2K^*) \to \cO_{l}(3)$ is zero. So, the central map factors through $\cO_{P'}(2)$, but $\Hom(F'',\cO(2f)) = 0$ is zero, hence $\phi$ is zero. 
\end{proof}

For now we proved that $F''$ is stable for with respect to suitable polarizations. 
In order to prove the main result we need to deal with other polarizations, the next result allows us to do so. 

\begin{Lem}
    Let $X$ be a projective \HK manifold, and $F$ a modular vector bundle on $X$ with Mukai vector $\vv$.
    Assume that $c_1(F) = 0$ and that $F$ is stable with respect to a $a(F)$-generic polarization $h$.
    Let $Y$ be a projective deformation of $X$, and let $h'$ be $a(F)$-generic. 
    Then there exists a vector bundle $F'$ on $Y$ which is $h'$-stable and with Mukai vector $\vv$.
\end{Lem}

\begin{proof}
    The proof is simply a refinement of the proof of \cite[Theorem 3.4]{markman21}, where we also keep track of the polarization. 
    Recall that, if $F$ is modular and $\omega \in \cK(X)$ is a K\"ahler class such that $F$ is $\omega$-stable, we can deform $F$ along the twistor line $\P^1_{\omega}$ spanned by $\omega$ by \cite[Theorem 3.19]{verbitsky99}.
    The assumption $c_1(F) = 0$ implies that even when we deform along a twistor line, the bundle remains untwisted. 
    The deformed bundle is stable with respect to the canonical K\"ahler class $\omega_t$ on every fiber $X_t$ of the twistor line. 
    Since the walls in the K\"ahler cone are defined by algebraic classes, we can find a K\"ahler class $\omega$ with the following properties:
    \begin{enumerate}
        \item The bundle $F$ is slope-stable with respect to $\omega$. 
        \item The twistor line $\P^1_\omega$ is generic, in the sense that the general element has trivial Picard group.
    \end{enumerate}
    We can do the same for $h'$, obtaining a K\"ahler class $\omega'$ on $Y$ with the same properties as above. 
    Now, we can deform $F$ along the twistor deformation $\cX_{\omega} \to \P^1_{\omega}$.
    In particular, if we choose a general element $X_1 \in  \P^1_{\omega}$, we obtain a modular bundle $F_1$ with the same invariants as $F$ living on $X_1$.
    Since $X_1$ has trivial Picard $F_1$ is stable with respect to any K\"ahler class by \cite[Lemma 6.15]{markman20}. 
    
    Denote by $X_{\mathrm{last}}$ a general element of $\P^1_{\omega'}$. Up to the choice of markings on $X_1$ and $X_{\mathrm{last}}$, we can connect them through a chain of twistor lines whose intersection points have trivial Picard group by \cite[Theorem 3.2 and 5.2e]{verbitsky96} (see also \cite[Theorem 6.14]{markman20}).
    The bundle $F_1$ deforms on this chain of twistor lines to a bundle $F_{\mathrm{last}}$ on $X_{\mathrm{last}}$, which again is stable with respect to any K\"ahler class. 
    Deforming $F_{\mathrm{last}}$ back along $\P^{1}_{\omega'}$ we get the desired bundle $F'$ on $Y$ which is $\omega'$-stable. 
    Since $\omega'$ and $h'$ live in the same open chamber, $F'$ is also $h'$-stable. 
\end{proof}

\begin{Rem}
    Notice that this works even if $Y = X$. 
    In this way, starting from a stable bundle, we can prove existence of a stable bundle for every generic polarization without the need to understand wall-crossing. 
\end{Rem}

\begin{proof}[Proof of Theorem \ref{thm:mainTheorem}]
The vector bundle $F''$ constructed above is atomic and stable with respect to a suitable polarization.
By Proposition \ref{prop:MukaiVector} we can twist $F''$ to obtain a stable atomic vector bundle $F_0$ with Mukai vector 
\[
v(F_0) = 5\left(1-\frac{3}{4}\sfq_2 + \frac{9}{32}\mathfrak{pt}\right).
\]
So $F_0$ satisfies the assumptions of the above lemma, and we can deform it to a vector bundle on any K\"ahler deformation of $X$, which is stable with respect to any $v(F_0)$-generic polarization.

The $\Ext$ algebra remains constant along these deformations by \cite[Proposition 6.3]{verbitsky08}.
The statement about the Ext groups is Corollary \ref{cor:YonedaSkewSymmetric}. 

To prove smoothness of the deformation space we argue as follows.
The main result in \cite{meazziniOnorati22} (or \cite[Theorem 6.1]{beckmann22}) gives formality of the algebra $\RHom(F_0,F_0)$. So the obstruction to lifting a first order deformation is the Yoneda square, which vanishes by Corollary \ref{cor:YonedaSkewSymmetric}. 
\end{proof}

\section{The moduli space}\label{sec:moduliSpace}
The sheaf $F''$ constructed in Section \ref{sec:ssreduction} is slope (hence Gieseker) stable for any suitable polarization on $M$.
Let $\fM$ be the irreducible component of the moduli space of Gieseker stable sheaves containing $F''$. 
Clearly, this is birational to a component of the moduli space of $F_0$.  
A priori, this component could depend on the choice of the curve $C \in |2H|$ used in the construction of $F''$, but we will show in Proposition \ref{prop:birationalMap} that it is not the case. 

%In this section we study the component $\fM$.
%We show that the smooth locus is endowed with a natural closed 2-form, and that $\fM$ is birational to a \HK variety of type OG10. 
%Our ultimate goal, which we fail to accomplish in this paper, is to show that $\fM$ is itself a \HK variety of type OG10.
\begin{Rem}
The component $\fM$ contains only Gieseker \emph{stable} sheaves; that is every Gieseker semistable sheaf in $\fM$ is also stable.
Since every $\P$-twist acts as the identity in cohomology, the Euler characteristic is unaffected by the semistable reduction. 
Example \ref{Ex:mixedexts} gives that $\chi(G) = -2$. It follows that 
\[
\chi(F'') = \chi(F) = \chi(G) + \chi(\cO_M) = 1,
\]
which is coprime with the rank, which guarantees Gieseker stability. 
\end{Rem}

We generalize the construction of Section \ref{sec:ssreduction} by considering certain line bundles of degree zero supported on $L \subset M$ constructed from line bundles of degree zero on curves in $|2H|$.
Line bundles of degree zero supported on curves in $|2H|$ are generic points of the singular moduli space $M_S(0,2H,-4)$.
By a celebrated result by O'Grady \cite{ogrady99} the singularities of $M_S(0,2H,-4)$ are symplectic. 
The symplectic resolution $\widetilde{M}_S(0,2H,-4)$ is a \HK variety of OG10 type, and the composition
\[
\widetilde{M}_S(0,2H,-4) \to M_S(0,2H,-4) \to |2H|
\]
is a Lagrangian fibration. 

\begin{Prop}\label{prop:birationalMap}
There is a birational map
\[
\theta: \widetilde{M}_S(0,2H,-4) \dashrightarrow \mathfrak{M} 
\]
\end{Prop}

\begin{proof}
Let $L_C$ be a line bundle of degree zero supported on a smooth curve $C \in |2H|$. 
Consider the $\Sigma_2$-equivariant line bundle
\[
 L_C \boxtimes L_C \in \Pic^0(C \times C).
\]
Being equivariant, it descends along the quotient $C \times C \to L$.
Denote by $\cL_C \in \Pic^0(L)$ descended line bundle. 
The rational map in the statement is then given by
\[
\widetilde{M}_S(0,2H,-4) \dashrightarrow \mathfrak{M}, \quad L_C \mapsto \Theta (\overline{\cL}_C)
\]
where $\Theta$ is the equivalence $P^{-1}_{\cO_M(2f)} \circ  P^{-1}_{\cO_M} \circ \Phi$, and ${\overline{\cL}_C}$ denotes the gluing of $\cL_C$ with the structure sheaf of $P'$. 
The fact that this is well defined on the open of smooth curves is the content of Theorem \ref{thm:mainTheorem}.
The map $\theta$ is injective because the restriction of the line bundle $\cL_C$ to the diagonal $\Delta \subset \Sym^2 C$ recovers the original $L_C$. 
\end{proof}

\subsection{Symplectic form}
Moduli spaces of stable sheaves on holomorphic symplectic surfaces are naturally equipped with a holomorphic symplectic form indudced by Serre duality. This was first observed by Mukai in \cite{mukai84}.
In the particular case of a smooth point in the moduli space $M_S(0,2H,-4)$, the tangent space to a point $[L_C]$ is identified to
\begin{equation*}\label{eq:tangentOG101}
    \Ext^1_S(L_C,L_C) \cong H^1(C,\C). 
\end{equation*}
Then, via this isomorphism the symplectic form is the cup product on $H^1(C,\C)$. 

More recently, Kuznetsov and Markushevich \cite{kuznetsovMark09} generalized this construction, and produced a closed 2-form the smooth locus of any moduli space of simple sheaves on an algebraic variety, which is not necessarily non-degenerate.
Here we briefly review the definition.

Recall that, for any $F \in \Db(M)$ and vector bundle $E$ on $M$, there is a trace map 
\[
\Tr_F : \Ext^k(F,F \otimes E) \to H^k(M,E).
\]
This was used in \cite{buchweitzFlenner03} to define the \emph{semiregularity map} for $F$
\[
\sigma: \Ext^2(F,F) \to \bigoplus_{p \geq 0} H^{p + 2}(M,\Omega^p_M), \quad \varphi \mapsto \Tr_F\left(\exp(-\At(F)) \circ \varphi\right),
\]
where $\At(F)$ is the \emph{Atiyah class} of $F$. 
Let $[F] \in \mathfrak{M}_{\mathrm{sm}}$ be a smooth point. Its tangent space is given by $\Ext^1(F,F)$ and the Yoneda pairing
\[
\Ext^1(F,F) \times \Ext^1(F,F) \to  \Ext^2(F,F), \quad (a,b) \mapsto a \circ b,
\]
is skew-symmetric. 
In \cite{kuznetsovMark09} the authors define for every $\omega \in H^*(M,\C)$, the following 2-form on $\mathfrak{M}_{\mathrm{sm}}$
\begin{equation}\label{eq:sympformViaSemireg}
   (a,b) \mapsto \int_M{\sigma(a \circ b) \cup \omega},
\end{equation}
and prove that it is closed.

We can write this form using the language of obstruction map.

\begin{Defi}
   For every $\eta \in HH^2(M)$ we define the 2-form 
\begin{equation}\label{eq:sympform}
    \alpha_{\eta}(a,b) \coloneqq \int_M{\Tr_F \left(\chi_F(\eta) \circ a \circ b \right) \cup \sigma^2_M}, 
\end{equation} 
where $\sigma_M \in H^0(M,\Omega^2_M)$ is the symplectic form. 
\end{Defi}

\begin{Lem}
    For every $\eta \in HH^2(M)$ there is an $\omega \in H^*(X,\C)$ such that 
\[
 \alpha_{\eta}(a,b) = \int_M{\sigma(a \circ b) \cup \omega}.
\]
In particular $\alpha_{\eta}$ is a closed 2-form on $\mathfrak{M}_{\mathrm{sm}}$.
\end{Lem}

\begin{proof}
Under the HKR isomorphism $HH^2(M) \cong HT^2(M)$, the obstruction map $\chi_F(-)$ becomes identified with
\[
HT^2(M) \to \Ext^2(F,F), \quad \eta \mapsto (\id_F \otimes \eta) \circ \exp(\At(F)),
\]
by \cite[Proposition 6.1]{toda09}.
We have 
\begin{align*}
\Tr_F(\chi_F(\eta) \circ a \circ b) &=  \Tr_F((\id_F \otimes \eta) \circ \exp(\At(F)) \circ a \circ b) \\ 
&= \eta \circ \Tr_F(\exp(\At(F)) \circ a \circ b),
\end{align*}
by linearity of the trace map. 
Up to changing signs of the graded pieces of $\eta$, this is equal to $\eta \circ \sigma(a \circ b)$. 
Then, by Poincaré duality there is a class $\omega \in H^*(X,\C)$, depending only on $\eta$, such that
\[
\int_M{(\eta \circ \sigma(a \circ b)) \cup \sigma^2_ M} = \int_M{\sigma(a \circ b) \cup \omega}.
\]
Closedness then follows from  \cite[Theorem 2.2]{kuznetsovMark09}.
\end{proof}

\begin{Rem}\label{rem:nonzero1obstructed}
On the $1$-obstructed locus, there is an $\eta \in HH^2(M)$ such that $\alpha_{\eta}$ is everywhere non-zero. Indeed by \cite[Lemma 4.2]{beckmann22} there is an inclusion 
\[
\Ker \chi_F \subseteq \Ker \chi_F^{\mathrm{coh}}, 
\]
which is an equality on the $1$-obstructed locus. 
In particular, $\Ker \chi_F$ depends only on $v(F)$. So, if $\chi_F(\eta) \neq 0$ for some 1-obstructed $F$, then is stays non-zero for all 1-obstructed sheaves in $\fM$.
\end{Rem}

It follows that, on the 1-obstructed locus there is a unique, up to a constant, 2-form of the form $\alpha_{\eta}$.
In Theorem \ref{thm:symplecticForm} we prove that the image of $\theta$ consists of 1-obstructed sheaves. 
We fix $\eta \in HH^2(M)$ such that $\alpha_{\eta}$ is everywhere non-zero on the image of $\theta$. 

%The following conjecture is what motivated our work. 
%A positive answer would make $\fM$ the first interesting example of a moduli space on a \HK fourfold, while at the same time providing a modular interpretation for OG10. 

\begin{Con*}%\label{con:OG10}
    The irreducible component $\fM$ is a smooth \HK variety of type OG10, with symplectic form given by $\alpha_{\eta}$. 
\end{Con*}

The biggest roadblock to prove the conjecture is understanding smoothness of $\mathfrak{M}$. 
In \cite[Conjecture B]{beckmann22} is conjectured that, at least for vector bundles, $1$-obstructedness implies smoothness. 
On the intersection of the 1-obstructed and the locally-free locus, the form $\alpha_{\eta}$ is also symplectic as explained in \cite[Section 8]{beckmann22}. 
This is a generalization of the classical result by Kobayashi \cite{kobayashi86}. 
In this section we make partial progress towards the conjecture above by proving Theorem \ref{thm:introBirational}.  

\subsection{Obstruction map}
In order to compare the symplectic forms on the source and target of $\theta$, it becomes necessary to understand the obstruction map for a sheaf in $\fM$.
Our goal is to show that it is one dimensional, and that under the isomorphism \eqref{eq:Ext2wedge} it is the line spanned by the dual of the cup product on $H^1(C,\C)$.

Let $C$ a smooth curve, and $\cL \in \Pic^0(L)$ any degree zero line bundle (they are all obtained as symmetrization of a line bundle on $C$).
The obstruction map is compatible with derived equivalences, so we need to compute
\[
\chi_{\overline{\cL}} : HH^2(M) \to \Ext^2(\overline{\cL},\overline{\cL}). 
\]
Via the HKR isomorphism there is a direct sum decomposition 
\[
HH^2(M) \cong H^0(M,\bwed^2 T_M) \oplus H^1(M,T_M) \oplus H^2(M,\cO_M). 
\]
The obstruction map vanishes on $H^2(M,\cO_M)$, because $Z \subset M$ is Lagrangian. 
The next Lemma deals with the first summand. 

\begin{Lem}\label{lem:obstructionNonCommutative}
Under $\chi_{\overline{\cL}}$, the image of $H^0(M,\bwed^2T_M)$ is contained in the image of $H^1(M,T_M)$.
\end{Lem}

\begin{proof}
The idea is to compare the obstruction map of $\overline{\cL}$ with the obstruction map of $F_0 = F \otimes \cO_M(-3f)$. 
By construction, the equivalence
\[
\Theta' = (- \otimes \cO(-3f)) \circ \Theta,
\]
maps $\overline{\cL}$ to $F_0$. 
There is an injective homomorphism
\[
\mu : HH^2(M) \to \widetilde{H}(X,\C), \quad \eta \mapsto m_{\eta}(\sigma), 
\]
where $m_{\eta}$ denotes the action of $\eta \in HH^2(M)$ on $\widetilde{H}(X,\C)$, see \cite[Section 6]{markman21}.
This action is induced by the action of the LLV algebra via \cite[Theorem A]{taelman21}.
If $\Phi$ is any equivalence, the map $\Phi^{\widetilde{H}}$ is a Hodge isometry compatibile with the action of the LLV algebra.
Hence $\mu$ intertwines the actions of $\Phi$, that is there is a commutative diagram
% https://q.uiver.app/?q=WzAsNCxbMCwwLCJISF4yKE0pIl0sWzEsMCwiXFx3aWRldGlsZGV7SH0oTSxcXEMpIl0sWzAsMSwiSEheMihNKSJdLFsxLDEsIlxcd2lkZXRpbGRle0h9KE0sXFxDKSJdLFswLDIsIlxcUGhpXntISH0iXSxbMSwzLCJcXFBoaV57XFx3aWRldGlsZGV7SH19Il0sWzAsMSwiXFxtdSJdLFsyLDMsIlxcbXUiXV0=
\[\begin{tikzcd}
	{HH^2(M)} & {\widetilde{H}(M,\C)} \\
	{HH^2(M)} & {\widetilde{H}(M,\C)}
	\arrow["{\Phi^{HH}}", from=1-1, to=2-1]
	\arrow["{\Phi^{\widetilde{H}}}", from=1-2, to=2-2]
	\arrow["\mu", from=1-1, to=1-2]
	\arrow["\mu", from=2-1, to=2-2]
\end{tikzcd}\]
A direct computation as in the proof of Proposition \ref{prop:MukaiVector} shows
\begin{align*}
&\Theta'^{HH} (\sigma^{\vee}) \in H^1(M,T_M) \oplus H^2(M,\cO_M) \text{ and,}\\
&\Theta'^{HH} (f) \in H^2(M,\cO_M).
\end{align*}
where $\sigma^{\vee}$ generates $H^0(M,\bwed^2T_M)$. 
The vector bundle $F_0$ deforms along every commutative deformation, so
\[
\chi_{F_0}|_{H^1(M,T_M)} \equiv 0, 
\]
which implies
\[
\chi_{\overline{\cL}}(\C \sigma^{\vee}) = \chi_{F_0}(\C \overline{\sigma}) = \chi_{\overline{\cL}}(\C f),
\]
where we are identifying $f \in H^1(M,\Omega^1_M)$ with its image in $H^1(M,T_M)$ via the isomorphism $\Omega^1_M \cong T_M$. 
\end{proof}

It is only left to compute the restriction of the obstruction map on $H^1(M,T_M)$. 
Recall that an element $\eta \in HH^2(M)$ represents a natural transformation $\eta: \id_M \to [2]$, and the obstruction map is the evaluation at an object. 
The naturality of $\eta$ provides a commutative triangle
% https://q.uiver.app/?q=WzAsMyxbMCwwLCJISF4yKFgpIl0sWzAsMSwiXFxFeHReMihcXGNPX0wsXFxjT19MKSJdLFsxLDEsIlxcRXh0XjIoXFxjT19aLFxcY09fWikiXSxbMCwxLCJcXGV0YV97XFxjT19MfSIsMl0sWzAsMiwiXFxldGFfe1xcY09fWn0iXSxbMSwyXV0=
\[\begin{tikzcd}
	{HH^2(M)} \\
	{\Ext^2(\cL,\cL)} & {\Ext^2(\overline{\cL},\overline{\cL})}
	\arrow["{\chi_{\cL}}"', from=1-1, to=2-1]
	\arrow["{\chi_{\overline{\cL}}}", from=1-1, to=2-2]
	\arrow[from=2-1, to=2-2]
\end{tikzcd}\]
where the horizontal map is the cokernel morphism of Lemma \ref{lem:extComputation}. 

\begin{Thm}\label{thm:symplecticForm}
The image of the obstruction map for $\overline{\cL}$ is one dimensional.
Under the isomorphism 
\[
\Ext^2(\overline{\cL},\overline{\cL}) \cong \bigwedge\nolimits^2 H^1(C,\C)
\]
it is spanned by the class representing the Poincar\'e pairing. 
In particular, $\theta$ maps the symplectic form to $\alpha_{\eta}$. 
\end{Thm}

\begin{proof}
By Lemma \ref{lem:obstructionNonCommutative} it suffices to compute the image of the restriction map
\[
H^2(M,\C) \to H^2(L,\C)
\]
as explained \cite[Remark 3.10]{markman21}, and map it into the quotient $H^2(L,\C)/H^0(C,\C).$ 
The restriction can be computed on the other birational model, that is
\[
H^2(S^{[2]},\C) \to H^2(\Sym^2C,\C) \cong H^2(C,\C) \oplus \bwed^2 H^1(C,\C). 
\]
The K\"unneth formula implies that the first summand in 
\[
H^2(S^{[2]},\C) \cong H^2(S,\C) \oplus \C\delta.
\]
maps to $H^2(C,\C)$.
By definition, the class $\delta$ maps to a multiple of the class of the diagonal $\Delta_C \subset \Sym^2 C$. 
To conclude note that, for every $\alpha,\beta \in H^1(C,\C)$ we have
\[
\int_{\Delta_C}{(\pi_1^*\alpha \wedge \pi^*_2\beta - \pi_1^*\beta \wedge \pi^*_2\alpha)|_{\Delta_C}} = 2\int_C{\alpha \wedge \beta}.
\]
Hence, the image of $[\Delta_C]$ in $\bigwedge\nolimits^2 H^1(C,\C)$ represents the Poincar\'{e} pairing.
Since $\alpha_{\eta}$ is defined by pairing two classes in $\Ext^1(\overline{\cL},\overline{\cL})$ with the image of the obstruction map, under the identification with $H^1(C,\C)$ induced by $\theta$, it corresponds to 
\[
\alpha_{\eta}(-,-) = \int_C{ - \wedge -}
\]
which is also the symplectic form on $\widetilde{M}_H(0,2H,-4)$. We conclude that $\theta$ preserves the 2-forms and that $\alpha_{\eta}$ is symplectic on the image. 
\end{proof}

\noindent
\textbf{Acknowledgments.}
This paper is part of my PhD thesis.
I want to thank my advisors Emanuele Macrì and Antonio Rapagnetta for their interest and support. 
I also want to thank Thorsten Beckmann, Kieran O'Grady and Eyal Markman for several fruitful and instructive conversations. 
I am partially supported by the European Research Council (ERC) under the European Union’s Horizon 2020 research and innovation programme (ERC-2020-SyG-854361-HyperK).

\end{document}